\documentclass[12pt]{article}   
   
\usepackage{latexsym,amsfonts,amsmath,epsfig,tabularx,amsthm,amssymb,enumitem,bm}
\usepackage{float}
\usepackage{mathtools}
\usepackage{authblk}
\usepackage[title,titletoc]{appendix}

\usepackage{microtype}
\usepackage[draft=false,colorlinks,bookmarksnumbered,linkcolor=black, citecolor=black]{hyperref}

\usepackage{aliascnt}

\newtheorem{theorem}{Theorem}[section]

\newaliascnt{lem}{theorem}
\newtheorem{lemma}[lem]{Lemma}

\aliascntresetthe{lem}

\newaliascnt{ass}{theorem}

\aliascntresetthe{ass}

\newaliascnt{prop}{theorem}
\newtheorem{prop}[prop]{Proposition}

\aliascntresetthe{prop}

\newaliascnt{cor}{theorem}

\aliascntresetthe{cor}

\newaliascnt{defi}{theorem}
\newtheorem{defi}[defi]{Definition}

\aliascntresetthe{defi}

\theoremstyle{definition}
\newaliascnt{ex}{theorem}

\aliascntresetthe{ex}

\newaliascnt{rem}{theorem}
\newtheorem{remark}[rem]{Remark}

\aliascntresetthe{rem}

\renewcommand{\d}{\,\mathrm{d}}											% Differential
\renewcommand*{\epsilon}{\varepsilon}                                   % modified epsilon
\renewcommand*{\rho}{\varrho}                                   		% modified rho
\newcommand*{\sep}{\; \vrule \;}                                        % set seperator: { x \in X \sep x=ay }
\newcommand*{\N}{\mathbb{N}}                                            % natural numbers IN
\newcommand*{\R}{\mathbb{R}}                                            % real numbers IR
\renewcommand*{\S}{\mathcal{S}}                                         % curved S
\newcommand*{\F}{\mathcal{F}}                                         	% curved F
\newcommand*{\D}{\mathcal{D}}                                         	% curved D
\newcommand*{\Co}{\mathcal{C}}                                         	% curved C
\newcommand*{\abs}[1]{\left| #1 \right|}                                % | . |
\newcommand*{\norm}[1]{\left\| #1 \right\|}                             % || . ||
\renewcommand{\tilde}[1]{ \widetilde{#1} }        						% tilde -> widetilde
\usepackage{tikz}
\usetikzlibrary{arrows,decorations.pathmorphing,decorations.pathreplacing,backgrounds,positioning,fit,petri, patterns}

\setlist{itemsep=2pt, topsep=2pt}

\newcommand{\cont}{\ensuremath{\mathcal{C}}}

\newcommand{\nrm}[1]{\ensuremath{\lVert #1 \rVert}}

\newcommand{\grklam}[1]{\big(#1\big)}

\newcommand{\ssgrklam}[1]{\bigg(#1\bigg)}

\newcommand{\ggklam}[1]{\big\{#1\big\}}
\newcommand{\sggklam}[1]{\Big\{#1\Big\}}

\newcommand{\geklam}[1]{\big [#1 \big ]}

\newcommand{\dx}{\mathrm d x}

\DeclareFontFamily{U}{matha}{\hyphenchar\font45}
\DeclareFontShape{U}{matha}{m}{n}{
      <5> <6> <7> <8> <9> <10> gen * matha
      <10.95> matha10 <12> <14.4> <17.28> <20.74> <24.88> matha12
      }{}
\DeclareSymbolFont{matha}{U}{matha}{m}{n}
\DeclareFontSubstitution{U}{matha}{m}{n}
\DeclareFontFamily{U}{mathx}{\hyphenchar\font45}
\DeclareFontShape{U}{mathx}{m}{n}{
      <5> <6> <7> <8> <9> <10>
      <10.95> <12> <14.4> <17.28> <20.74> <24.88>
      mathx10
      }{}
\DeclareSymbolFont{mathx}{U}{mathx}{m}{n}
\DeclareFontSubstitution{U}{mathx}{m}{n}
\DeclareMathDelimiter{\vvvert}{0}{matha}{"7E}{mathx}{"17}
\title{
On the limit regularity in Sobolev and Besov scales related to approximation theory
}
\date{}
 
\author{
Petru A. Cioica-Licht\footnote{Universit\"at Duisburg-Essen,  Fakult\"at Mathematik, AG Stochastische Analysis, 45117 Essen \textsc{and} University of Otago, Department of Mathmatics and Statistics, P.O. Box 56, Dunedin 9054, New Zealand. Email: petru.cioica-licht@uni-due.de}
\qquad
Markus Weimar\footnote{\emph{Corresponding author}. Ruhr University Bochum, Faculty of Mathematics, Research Group Numerics, Universit\"atsstra{\ss}e 150,
44801 Bochum, Germany. Email: markus.weimar@rub.de. }
}

\begin{document}   
\maketitle   

\vspace*{-2em}

\begin{center}
\emph{Dedicated to Prof.~Dr.~Stephan Dahlke on the occasion of his 60th birthday}
\end{center}

\smallskip

\begin{abstract}
\noindent 
We study the interrelation between the limit $L_p(\Omega)$-Sobolev regularity $\overline{s}_p$ of (classes of) functions on bounded Lipschitz domains $\Omega\subseteq\R^d$, $d\geq 2$, and the limit regularity $\overline{\alpha}_p$ within the corresponding adaptivity scale of Besov spaces $B^\alpha_{\tau,\tau}(\Omega)$, where $1/\tau=\alpha/d+1/p$ and $\alpha>0$ ($p>1$ fixed). 
The former determines the convergence rate of uniform numerical methods, whereas the latter corresponds to the convergence rate of best $N$-term approximation. 
We show how additional information on the Besov or Triebel-Lizorkin regularity may be used to deduce upper bounds for $\overline{\alpha}_p$ in terms of $\overline{s}_p$ simply by means of classical embeddings and the extension of complex interpolation to suitable classes of quasi-Banach spaces due to Kalton, Mayboroda, and Mitrea (Contemp.\ Math.~445).
The results are applied to the Poisson equation, to the $p$-Poisson problem, and to the inhomogeneous stationary Stokes problem.
In particular, we show that already established results on the Besov regularity for the Poisson equation are sharp.


\end{abstract}

\section{Introduction}\label{sec:Intro}
The convergence rate of approximation methods strongly depends on the regularity of the target function.
In particular,  the convergence rate of the best $N$-term approximation for a function $f\colon\Omega\to\R$ on a bounded Lipschitz domain $\Omega\subseteq\R^d$, $d\in\N$, is intimately related to its regularity in the scale of Besov spaces
\begin{equation}\label{eq:scale}
	B^\alpha_{\tau,\tau}(\Omega),\qquad\frac{1}{\tau}=\frac{\alpha}{d}+\frac{1}{p},\quad \alpha>0,\tag{$*$}
\end{equation} 
whereas the convergence of an approximation method based on uniform refinements depends on the regularity in the scale $W^s_p(\Omega)$, $s>0$, of Sobolev spaces; here, $1<p<\infty$ is fixed and the approximation error is measured in $L_p(\Omega)$. 
Roughly speaking, if (and only if) the Besov regularity of the target function in the scale~\eqref{eq:scale} is strictly higher than its corresponding Sobolev regularity, a higher convergence rate may be achieved by switching from uniform refinement strategies to more sophisticated adaptive wavelet or finite element schemes.
We refer to~\cite{DahDahDeV1997, DeV1998, GasMor2014} and to the references therein for details and sufficient assumptions for such statements. % above.
Definitions of the relevant function spaces are provided in the appendix.

The Sobolev regularity of solutions to elliptic partial differential equations on non-smooth domains may be very limited, even if the forcing terms are infinitely smooth.
Upper bounds for 
\begin{equation}\label{eq:def:sbar}
	\overline{s}_p
	:=\overline{s}_p(S(\Omega))
	:=\sup\ggklam{s>0 \sep S(\Omega)\subseteq W^s_p(\Omega)},
\end{equation}
where $S(\Omega)\subseteq L_p(\Omega)$ is a suitably chosen set of solutions to various instances of  elliptic equations, can be found, for instance, in~\cite{Cos2019, FabMenMit1998, Gri1985, JK95, MazRos2010, Sav1998}.
To mention an example, there exist bounded $\cont^1$ domains $\Omega\subseteq\R^d$ such that if we define $S(\Omega)$ to be the set of all solutions to the Poisson equation with zero Dirichlet boundary conditions and right hand sides $f\in\Co^\infty(\overline{\Omega})$, then $\overline{s}_p(S(\Omega))=1+1/p$, see~\autoref{sect:Poisson} for details.
Similar results for (stochastic) evolution equations can be found, e.g., in~\cite{Gri1992, Lin2014}.
At the same time, we know that the solution to most of the equations in the aforementioned references may have higher regularity $\alpha>\overline{s}_p$ in the scale~\eqref{eq:scale}, see, e.g.,~\cite{CDK+,Dah1999,DahDeV1997,DahDieHar+2014,DaSi08,DaSi13,DahWei2015,Eck2015,EckCioDah2016,Han2015}.
For instance, in the example above, it is known that 
\begin{equation*}
	S(\Omega)\subseteq B^\alpha_{\tau,\tau}(\Omega),\qquad \frac{1}{\tau}=\frac{\alpha}{d}+\frac{1}{p},\quad\text{for all}\quad 0<\alpha<\bigg(1+\frac{1}{p}\bigg)\frac{d}{d-1},
\end{equation*}
see~\cite{DahDeV1997}.
The higher Besov regularity justifies the development of adaptive numerical methods for (stochastic) partial differential equations.
However, to the best of our knowledge, \emph{no upper bound at all} for the regularity in the scale~\eqref{eq:scale}, i.e., for
\begin{equation}\label{eq:def:alphabar}
	\overline{\alpha}_p
	:=\overline{\alpha}_p(S(\Omega))
	:=\sup\sggklam{\alpha>0 \sep S(\Omega)\subseteq B^\alpha_{\tau,\tau}(\Omega), \;\; \frac{1}{\tau}=\frac{\alpha}{d}+\frac{1}{p}}
\end{equation}
can be found in the literature; here, $\sup\emptyset:=-\infty$.
Thus, in many settings, we do know that there is the possibility to outperform uniform methods by adaptive refinement strategies but we do not know how high the convergence rate of these methods can maximally get.
 Note that the cases $\overline s_p=\infty$, resp.\ $\overline \alpha_p=\infty$, are explicitly allowed and indeed occur already in the most basic examples; see, e.g., Remark~\ref{rem:grisvard}. 

In this paper we study the interrelation between the limit regularity indices $\overline{s}_p$ and $\overline{\alpha}_p$.
In \autoref{sect:mainresult} we prove an abstract result showing for arbitrary sets $S(\Omega)\subseteq L_p(\Omega)$ how additional information about the Besov or Triebel-Lizorkin regularity of all $u\in S(\Omega)$ can be used to deduce \emph{upper bounds} for $\overline{\alpha}_p$ in terms of $\overline{s}_p$ simply by means of the extension of complex interpolation to suitable classes of quasi-Banach spaces from~\cite{KMM07} and classical embeddings.
We apply this result in \autoref{sect:example} to the Poisson equation, the $p$-Poisson problem, and the inhomogeneous stationary Stokes equation. In particular,  we show that under fairly natural assumptions, already established positive results on the Besov regularity of the solution to the Poisson equation in the scale~\eqref{eq:scale}  are actually sharp.
Before we start, we introduce some notation and comment on so-called DeVore-Triebel diagrams, which we will use in order to visualize results.

\medskip

\noindent\textbf{Notation.}  
Throughout this manuscript, $\Omega$ denotes a bounded Lipschitz domain in $\R^d$ for some $d\in\N$.
For $0<p< \infty$, by $L_p(\Omega)$ we denote the space of all (equivalence classes of) Lebesgue-measurable, scalar-valued functions satisfying $\nrm{u\sep L_p(\Omega)}^p:=\int_{\Omega} \abs{u(x)}^p\,\dx<\infty$, while $L_\infty(\Omega)$ is the space of all (equivalence classes of) Lebesgue-measurable, Lebesgue-almost everywhere bounded scalar-valued functions on $\Omega$.
Moreover,  $B^s_{p,q}(\Omega)$ and $F^s_{p,q}(\Omega)$ stand for the Besov and Triebel-Lizorkin spaces, respectively, with smoothness parameter $s\in\R$, integrability parameter $p\in(0,\infty]$ (with $p<\infty$ for Triebel-Lizorkin spaces) and microscopic parameter $q\in(0,\infty]$.
The corresponding spaces $B^s_{p,q}(\partial\Omega)$ and $F^s_{p,q}(\partial\Omega)$ on the boundary~$\partial\Omega$ of the domain $\Omega$ are defined as in~\cite{MitWri2012}. 
For $1<p<\infty$, by~$W^s_p(\Omega)$ we denote the $L_p(\Omega)$-Sobolev space of order $s\in\R$. 
For two quasi-normed spaces $X$ and $Y$, we write $X\hookrightarrow Y$ if $X$ is continuously and linearly embedded in $Y$ and $[X,Y]_{\theta}$ stands for the complex interpolation space of the pair $(X,Y)$ with parameter $\theta\in(0,1)$.
Precise definitions and relevant interpolation and embedding properties of Besov, Triebel-Lizorkin, and Sobolev spaces are collected in \autoref{sect:appendix}.

Throughout, the letter $C$ is used to denote a finite positive constant that may differ from one appearance to another, even in the same chain of inequalities. 
Moreover, we adopt the usual conventions $1/\infty:=0$ and $1/0:=\infty$. 

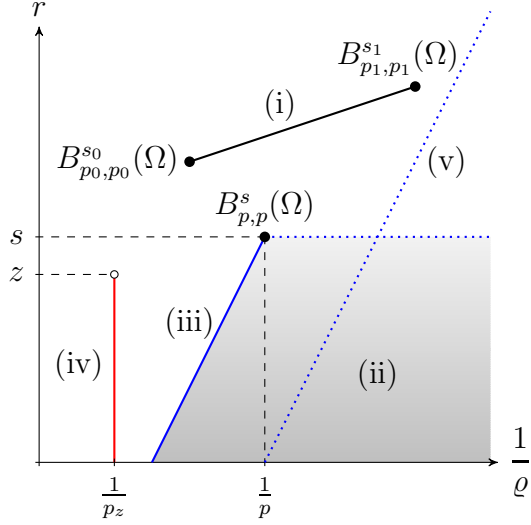
\begin{figure}
\begin{center}
\begin{tikzpicture}
% The shaded area: implyed by $B^s_{p,p}$
\shade [left color=gray!50, right color=gray!10, shading angle = 180] (1.5,0) -- (3,3) -- (6,3) -- (6,0);
\draw (4.5,0.8) node[above] {(ii)};

% The coordinate axes
\draw [->,>=stealth'] (0,-0.05) -- (0,5.8) node[above] {$r$};
\draw [->,>=stealth'] (-0.05,0) -- (6.1,0) node[right] {$\displaystyle\frac{1}{\rho}$};

% The linear appoximation line at for p=2 incl. a special point
\draw [red, thick] (1,0)  -- (1,2.5);
\draw (1,0.05) -- (1,-0.05) node[below] {$\frac{1}{p_z}$};
\draw (-0.3,2.5) node {$z$};
\draw [dashed] (-0.05,2.5) -- (1,2.5); 
\draw (1,2.5) node[circle,draw=black, fill=white, inner sep=1pt]  {};
\draw (1,1.3) node[left] {(iv)};

% The delimitation of the embedding area
\draw [blue,thick] (1.5,0) -- (3,3);
\draw (2,1.5) node[above]{(iii)};
\draw [blue, dotted, thick] (3,3) -- (6,3);

% The point (1/p,s) representing $B^s_{p,p}(\domain)$ (samt Achsenbeschriftung)
\draw (3,-0.05) node[below] {$\frac{1}{p}$};
\draw [dashed] (-0.05,3) -- (3,3);
\draw (-0.3,3) node {$s$};
\fill (3,3) circle (2pt) node[anchor=south] {$B^s_{p,p}(\Omega)$};
\draw [dashed] (3,-0.05) -- (3,3);

% Sobolev embedding line
\draw [dotted,blue,thick] (3,0) -- (6,6);
\draw (5,4) node[anchor=west] {(v)};
% interpolation
\fill (5,5) circle (2pt) node[anchor=south] {$B^{s_1}_{p_1,p_1}(\Omega)$\phantom{nn}};
\fill (2,4) circle (2pt) node[anchor=east] {$B^{s_0}_{p_0,p_0}(\Omega)$};
\draw [thick] (2,4) -- (5,5);
\draw (3.2,4.4) node[above] {(i)};
\end{tikzpicture}
\caption[DeVoreTriebel00]{\begin{tabular}[t]{l}Visualization of Besov spaces on bounded Lipschitz \\ domains $\Omega\subseteq\R^d$ in a DeVore-Triebel diagram.\end{tabular}}\label{fig:DeVoreTriebel00}
\end{center}
\end{figure}

\medskip

\noindent\textbf{DeVore-Triebel diagrams.} We are going to use so-called DeVore-Triebel diagrams in order to visualize results. 
In those $(1/p,s)$-diagrams, we identify every point $(1/p,s)\in [0,\infty)\times \R$ with the Besov space $B^s_{p,p}(\Omega)$.
%, while $(1/p,0)$ stands for $L_p(\Omega)$. 
Many embedding and interpolation results for Besov spaces can then be visualized in a very convenient way (see~\autoref{fig:DeVoreTriebel00}):
\newpage
\begin{itemize}
\item Besov spaces form scales of (generalized) complex interpolation spaces, see \autoref{prop:interpol}. As a consequence, if $f\in B^{s_i}_{p_i,p_i}(\Omega)$ for $i=0,1$,  then $f\in B^{\tilde{s}}_{\tilde{p},\tilde{p}}(\Omega)$ for all $(1/\tilde{p},\tilde{s})$ on the line segment between $(1/p_0,s_0)$ and $(1/p_1,s_1)$; see (i) in~\autoref{fig:DeVoreTriebel00}.
\item If $f\in B^{s}_{p,p}(\Omega)$ for some $0<p<\infty$ and $s\in\R$, then, by \autoref{prop:embeddings}\ref{it:embedd:Sob:eps}, $f$ is contained in all the Besov spaces represented by the points $(1/\tilde{p},\tilde{s})\in [0,\infty)\times\R$ with $\tilde{s}< s-d\,\max\ggklam{ 1/p-1/\tilde{p},\,0}$; see the shaded area~(ii) in~\autoref{fig:DeVoreTriebel00}. Moreover, by \autoref{prop:embeddings}\ref{it:embedd:sob:sharp}, it is contained in all Besov spaces represented by the points $(1/\tilde{p},\tilde{s})\in (0,1/p)\times\R$ with $\tilde{s}=s-d\,\grklam{1/p-1/\tilde{p}}$; see~(iii) in~\autoref{fig:DeVoreTriebel00}.
\item If $f\in A^{z}_{p_z,q_z}(\Omega)$ for some $A\in \{B,F\}$, $z\in \R$ and $0<p_z,q_z\leq \infty$ (with finite $p_z$ if $A=F$), then, by~\autoref{prop:embeddings}\ref{it:embedd:Sob:eps}, $f$ is contained in all Besov spaces represented by the ray $\{(1/p_z,\tilde{s}) \sep \tilde{s}<z\}$; see (iv) in~\autoref{fig:DeVoreTriebel00}.
\end{itemize}
Moreover, in such a diagram, for $1<p<\infty$, the scale~\eqref{eq:scale} is represented by the so-called \emph{$L_p(\Omega)$-Sobolev embedding line} 
\begin{equation}\label{eq:Lp-sob-embedding}
	\left\{ \left( \frac{1}{\tau}, \alpha \right) \in (0,\infty)^2 \sep \frac{1}{\tau}=\frac{\alpha}{d} + \frac{1}{p}\right\},
\end{equation}
see (v) in~\autoref{fig:DeVoreTriebel00}.

\section{Main result}\label{sect:mainresult}
In this section we analyze how additional information about the Besov or Triebel-Lizorkin regularity may be used in order to derive upper bounds for $\overline{\alpha}_p$ in terms of $\overline{s}_p$ simply by means of complex interpolation and classical embedding theorems; here and in the sequel, $\overline{s}_p$ and $\overline{\alpha}_p$  are defined as  in~\autoref{sec:Intro}, see~\eqref{eq:def:sbar} and~\eqref{eq:def:alphabar}, respectively.
We prove the following main result.

\begin{theorem}\label{thm:main:new}
	For $d\in\N$ let $\Omega\subseteq\R^d$ be a bounded Lipschitz domain. 
	Moreover, let $1<p<\infty$ and let $S(\Omega)\subseteq L_p(\Omega)$ be such that $0<\overline{s}_p=\overline{s}_p(S(\Omega))\leq \infty$.
	Assume that for some $z\in\R$ and some $p<p_z\leq\infty$,  $S(\Omega)\subseteq B^s_{p_z,p_z}(\Omega)$ for all $s<z$. 
	Then
	\begin{equation}\label{eq:z:s:alpha}
		z\leq \displaystyle \overline{s}_p\leq \overline{\alpha}_p.
	\end{equation}
	If additionally 
	\begin{equation*}
		z > \mu  
		:=\mu(p_z,p,\overline{s}_p,d)
		:=\overline{s}_p-d\,\ssgrklam{\frac{1}{p}-\frac{1}{p_z}},
\end{equation*}
then
\begin{equation}\label{eq:alpha:bound}
	\overline{\alpha}_p\leq \overline{s}_p+\overline{s}_p\cdot\frac{\overline{s}_p-z }{z-\mu }=\overline{s}_p\cdot\frac{\overline{s}_p-\mu }{z-\mu }.
\end{equation}
\end{theorem}

Before we give a proof of this theorem, let us make some remarks. We start with a sufficient condition for the additional regularity assumption. 

\begin{remark}\label{rem:sufficient}
Let $0<p_z<\infty$ and $z\in\R$. Then, by classical embedding theorems for Besov and Triebel-Lizorkin spaces, as collected in~\autoref{prop:embeddings}, the assertion
\begin{equation*}
	S(\Omega)\subseteq A^z_{p_z,q_z}(\Omega) \quad \text{for some}\quad A\in\{B,F\}\quad \text{and}\quad 0<q_z\leq \infty,
\end{equation*}
is sufficient for
\begin{equation*}
	S(\Omega)\subseteq B^s_{p_z,p_z}(\Omega)\quad\text{for all }\quad s<z.
\end{equation*}
Moreover, so is
\begin{equation*}
	S(\Omega)\subseteq A^s_{p_z,q_z}(\Omega) \quad \text{for some}\quad A\in\{B,F\},\quad 0<q_z\leq \infty, \quad \text{and all}\quad s<z.
\end{equation*}
If $A=B$, then these implications also hold for $p_z=\infty$.
\end{remark}

\begin{remark}
In principle, $S(\Omega)$ could be any subset of some Besov/Triebel-Lizorkin space. But even if we restrict ourselves to solution sets for operator equations, there are several different interpretations:
On the one hand, we may think of \emph{one particular} problem given by a fixed operator $L$ acting on functions defined on a fixed domain $\Omega$ with fixed right-hand side and fixed initial/boundary conditions if necessary.
Then $S(\Omega)$ only contains solutions for this particular situation and we probably even have $\# S(\Omega)=1$ such that $\overline{s}_p$ and $\overline{\alpha}_p$ describe smoothness properties of \emph{one particular function}. 
On the other hand, we may also think of solution sets for \emph{classes} of problems such as, e.g.,
\begin{itemize}
	\item[(i)] a fixed equation (like the Poisson equation $\Delta u = f$ with zero Dirichlet boundary condition $u_{|\partial\Omega}=0$) on a fixed domain $\Omega$  (e.g., the standard L-shape domain in $d=2$) with \emph{variable} right-hand side from a certain class of functions (e.g., arbitrary $f\in L_2(\Omega)$), or
	\item[(ii)] a class of operator equations (e.g., all linear, second order PDEs with smooth coefficients) on a fixed domain $\Omega$ with, say, smooth right-hand sides,
\end{itemize}
and so forth.
Since in this case $S(\Omega)$ collects \emph{all} functions which solve at least one admissible problem instance, here $\overline{s}_p$ and $\overline{\alpha}_p$ describe \emph{lower bounds} for the regularity of solutions to the hardest possible problem in the respective class. For example, $u^*\equiv 0$ solves the problem described in (i) for $f\equiv 0$. Hence, $u^*\in S(\Omega)$ and $u^*\in \bigcap_{s>0} W^s_2(\Omega)$, but $\overline{s}_2=5/3<\infty$, see also Remark~\ref{rem:grisvard} below. 

We could even go one step further and consider classes of problems like
\begin{itemize}
	\item[(iii)] a fixed equation considered on a class of domains (e.g., all bounded $\cont^1$ domains) with certain restrictions on the right-hand side and/or on initial/boundary conditions.
\end{itemize}
However, then the notation would get more complicated such that in the sequel we restrict ourselves to the cases mentioned above.
\end{remark}

\begin{remark}\label{rem:main}
Throughout this remark, we assume that we are in the setting of \autoref{thm:main:new}.
\begin{enumerate}[align=right,label=\textup{(\roman*)}] 
	\item\label{rem:main:q} Note that, due to standard embeddings of Besov and Triebel-Lizorkin spaces (as provided in \autoref{prop:embeddings} in the appendix), 
for $A\in\{B,F\}$ and $0<q\leq\infty$ we have
\begin{align*}
	\overline{s}_p
	=\overline{s}_p(S(\Omega))
	=\overline{s}_{p,q,A}(S(\Omega))
	:=\sup\ggklam{s>0 \sep S(\Omega)\subseteq A^s_{p,q}(\Omega)}.
\end{align*}
That is, the limit regularity index $\overline{s}_{p,q,A}$ does not depend on the microscopic parameter~$q$, nor on the type of the spaces $A$ (Besov vs.\ Triebel-Lizorkin). Moreover, it coincides with $\overline{s}_p$ defined in~\eqref{eq:def:sbar}.
In particular,
\begin{equation}\label{eq:bars:p:B}
	\overline{s}_p
	=\overline{s}_{p,p,B}(\Omega)=\sup\ggklam{s>0 \sep S(\Omega)\subseteq B^s_{p,p}(\Omega)}
\end{equation}
and also
\begin{equation*}
	\overline{s}_p
	=\overline{s}_{p,\infty,B}(S(\Omega))
	=\sup\ggklam{s>0 \sep S(\Omega)\subseteq B^s_{p,\infty}(\Omega)},
\end{equation*}
where the latter quantity is defined by means of the slightly larger Besov spaces~$B^s_{p,\infty}(\Omega)$ which coincide with the approximation spaces $\mathcal{A}_\infty^{s/d}(L_p(\Omega))$ w.r.t.\ non-adaptive algorithms based on uniform refinement, see, e.g.,~\cite{DeV1998} for details. 

	\item\label{it:rem:2} Due to the generalization of Sobolev's embedding theorem to Besov spaces (as presented in~\autoref{prop:embeddings}\ref{it:embedd:sob:sharp}), a space $B^\alpha_{\tau,\tau}(\Omega)$ from the adaptivity scale~\eqref{eq:scale} is embedded into every other space $B^{\alpha_0}_{\tau_0,\tau_0}(\Omega)$, $1/\tau_0=\alpha_0/d+1/p$, 
	%from the same scale~\eqref{eq:scale} with $\alpha_0<\alpha$.
	 from the same scale with $0\leq \alpha_0 < \alpha$.  
	However, as a consequence of the sharpness of Sobolev embeddings, the space $B^\alpha_{\tau,\tau}(\Omega)$ is not embedded in $A^s_{p,q}(\Omega)$ for any $A\in\{B,F\}$, $0<q\leq\infty$, and $s>0$, as this 
	 combined with \autoref{prop:embeddings}\ref{it:embedd:Sob:eps} 
	would contradict the `only if' part of~\autoref{prop:embeddings}\ref{it:embedd:sob:sharp}.
	%As a consequence, 
	 Therefore,  
	it is not possible to obtain a non-trivial upper bound for $\overline{\alpha}_p$ in terms of $\overline{s}_p$ without further assumptions on $S(\Omega)$.

	\item\label{rem:main:proof} In~\autoref{fig:DeVoreTriebel01} we use a DeVore-Triebel diagram to visualize our upper bound~\eqref{eq:alpha:bound} for~$\overline{\alpha}_p$ and the corresponding proof idea, given that $\overline{s}_p<\infty$.
%	; relevant details on DeVore-Triebel diagrams can be found at the end of the introduction above.
The bound $\overline{s}_p\cdot(\overline{s}_p-\mu )/(z-\mu )$ in~\eqref{eq:alpha:bound} is precisely the ordinate of the intersection point of the (dashed) line through $(1/p_z,z)$ and $(1/p,\overline{s}_p)$ with the $L_p(\Omega)$-Sobolev embedding line~\eqref{eq:Lp-sob-embedding}.
Therefore, by elementary geometry, for every $\alpha>\overline{s}_p\cdot(\overline{s}_p-\mu )/(z-\mu )$, there exists $\tilde{z}<z$, such that the (solid) line through $(1/p_z,\tilde{z})$ and $(\alpha/d+1/p,\alpha)$ contains a point~$(1/p,s)$ for some $s>\overline{s}_p$.
Since $S(\Omega)\subseteq B^{\tilde{z}}_{p_z,p_z}(\Omega)$ for all $\tilde{z}<z$, the claim $S(\Omega)\subseteq B^\alpha_{\tau,\tau}(\Omega)$ for such an $\alpha$ would thus contradict the maximality of $\overline{s}_p$, see also~\eqref{eq:bars:p:B}.

\begin{figure}
\begin{center}
\begin{tikzpicture}

% The coordinate axes
\draw [->,>=stealth'] (0,-0.05) -- (0,5.8) node[above] {$r$};
\draw [->,>=stealth'] (-0.05,0) -- (5.5,0) node[right] {$\displaystyle\frac{1}{\rho}$};

% The linear appoximation line at for p=2 incl. a special point
\draw [red, thick] (0.5,0)  -- (0.5,2.5);
\draw (0.5,-0.05) -- (0.5,0.05) node[below] {$\frac{1}{p_z}$};
\draw (1.5,-0.05) -- (1.5,0.05) node[below] {$\frac{1}{p}$};
\draw (-0.3,2.5) node {$z$};
\draw [dashed] (-0.05,2.5) -- (0.5,2.5); 

% The $L_p$-Sobolev embedding line
\draw [blue,thick] (1.5,0) -- (3.2,3.4);
\draw [blue,thick,dotted] (3.2,3.4)--(3.5,4);
\draw [dashed](-0.05,3.4)--(3.2,3.4);
\draw [dashed](-0.05,4.81)--(3.8,4.81);
\draw (-0.3,3.5) node {$\overline{\alpha}_p$};
\draw (-0.3,4.81) node {$\alpha$};
\draw [blue, dotted, thick] (3.5,4) -- (4.5,6);
\draw (4.5,5.7) node[right] {$\displaystyle\frac{1}{\tau}=\frac{\alpha}{d}+\frac{1}{p}$} ;

% B^s_{p,p}-Linie
\draw [thick, blue] (1.5,0) -- (1.5,3) ;
\draw [thick, blue, dotted] (1.5,3)--(1.5,6) ;
\draw [dashed] (-0.05,3) -- (1.5,3);
\draw (-0.3,3) node {$\overline{s}_p$};

% Sobolev embedding at \overline{s}
\draw [thick, blue, dotted](0.25,0.5)--(1.5,3);
\draw [thick, blue, dashed](0.5,2.5)--(3.5,4);
\draw  (0.5,2.35)--(3.9,4.81);
\draw [->, >=stealth'](-0.6,4.2)--(-0.05,4);
\draw (-1.7,4.2) node {$\displaystyle\overline{s}_p\cdot\frac{\overline{s}_p-\mu}{z-\mu}$};
\draw [dashed] (-0.05,4)--(3.5,4);
\draw (1.5,3) node[circle,draw=black, fill=white, inner sep=1pt]  {};
\draw (0.5,2.5) node[circle,draw=black, fill=white, inner sep=1pt]  {};
\draw (3.2,3.4) node[circle,draw=black, fill=white, inner sep=1pt]  {};
\draw [->, >=stealth'](-0.3,0.7)--(0.35,0.7);
\draw (-0.3,0.7) node[left] {$\mu(\cdot,p,\overline{s}_p,d)$};
\draw [dashed](-0.05,2.35)--(0.5,2.35);
\draw [->, >=stealth'](-0.5,2)--(-0.05,2.35);
\draw (-0.5,2) node[left] {$\tilde{z}$};
\end{tikzpicture}
\caption[DeVoreTriebel01]{\begin{tabular}[t]{l}Visualization of statement and proof of Assertion~\eqref{eq:alpha:bound}  \\ from~\autoref{thm:main:new}  in a DeVore-Triebel diagram.\end{tabular}}\label{fig:DeVoreTriebel01}
\end{center}
\end{figure}
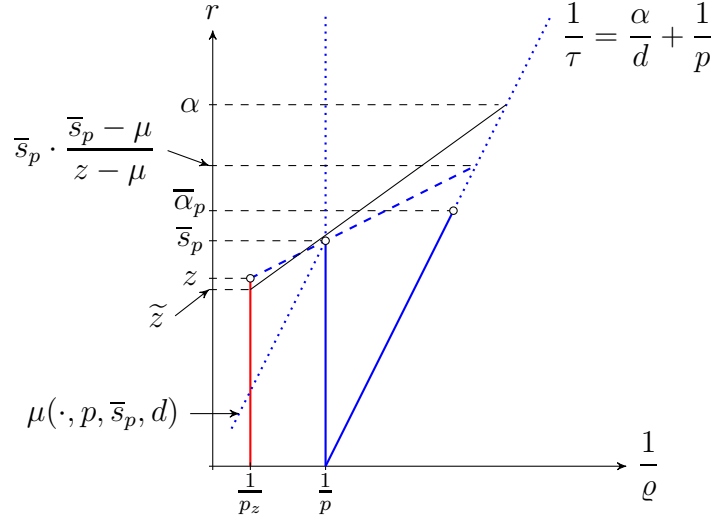

\item\label{rem:main:fail} The proof idea above obviously fails if $z\leq \mu(p_z,p,\overline{s}_p,d)$, i.e., if the point $(1/p_z,z)$ is below or exactly on the Sobolev embedding line $\ggklam{(1/\tilde{p},\tilde{s}) \sep \tilde{s}=\overline{s}_p-d\,(1/p-1/\tilde{p})}$ through~$(1/p,\overline{s}_p)$. In this case the line through~$(1/p_z,z)$ and $(1/p,\overline{s}_p)$ does not intersect with the corresponding $L_p(\Omega)$-Sobolev embedding line~\eqref{eq:Lp-sob-embedding}.

Actually, it is clear that we cannot even expect to obtain a non-trivial bound on $\overline{\alpha}_p$ if we only know that 
$S(\Omega)\subseteq B^s_{p_z,p_z}(\Omega)$ for all $s<z\leq \mu(p_z,p,\overline{s}_p,d)$, since this is already implied by Sobolev's embedding theorem (see~\autoref{prop:embeddings}\ref{it:embedd:Sob:eps}).
Thus, assuming this does not add any additional information about $S(\Omega)$ and we cannot expect to be able to establish a non-trivial bound on $\overline{\alpha}_p$, see also~\ref{it:rem:2} above.
In the limiting case, i.e., if $z=\mu$, then assuming that $S(\Omega)\subseteq A^z_{p_z,q_z}(\Omega)$ for some $A\in\{B,F\}$ and $0<q_z\leq \infty$ as in~\autoref{rem:sufficient} may or may not constitute an additional assumption on~$S(\Omega)$. However, also in this case it is not possible to establish a non-trivial bound for~$\overline{\alpha}_p$. Counterexamples can easily be constructed in terms of standard representatives of Besov and Triebel-Lizorkin spaces; see, in particular, \cite[Lemma~2.3.1.1]{RunSic1996}.

\item\label{rem:main:lower} The proof technique described in~\ref{rem:main:proof} above may also be used in order to derive, for instance,
\begin{itemize}
	\item the lower bound 
\begin{equation*}
	\tilde{s}_p:=\overline{\alpha}_p\cdot\frac{z+d\,(1/p-1/p_z)}{\overline{\alpha}_p+d\,(1/p-1/p_z)}  
\end{equation*} 
for $\overline{s}_p$, provided we are given $\overline{\alpha}_p>0$ and $S(\Omega)\subseteq A^z_{p_z,q_z}(\Omega)$ for some $A\in\{B,F\}$, $p_z>p$, $0<q_z\leq\infty$, and $z\in\R$, or

	\item an upper bound for $\overline{s}_{\widehat{p}}$ for some $\widehat{p}>p$, given $\overline{s}_p$, as well as $S(\Omega)\subseteq A^z_{p_z,q_z}(\Omega)$ for some $A\in\{B,F\}$, $z>\overline{s}_p$, and $p_z<p$.
\end{itemize}
In~\autoref{sect:Poisson}, we are going to use the latter in order to determine $\overline{s}_p$, $1<p<\infty$, for the Poisson equation with smooth right-hand sides and zero Dirichlet boundary conditions on a bounded $\cont^1$ domain constructed by Jerison and Kenig~\cite{JK95}.

	\item Further assumptions of the type $S(\Omega)\subseteq A^{\tilde{z}}_{\tilde{p}_z,\tilde{q}_z}(\Omega)$ for some $A\in\{B,F\}$, as well as $1<p<p_z<\tilde{p}_z\leq \infty$ (with finite $\tilde{p}_z$ if $A=F$), $0<\tilde{q}_z\leq\infty$, and $\tilde{z}\in\R$ lead to an improvement of the upper bound for $\overline{\alpha}_p$ by means of the proof technique described in~\ref{rem:main:proof} only if the point $(1/\tilde{p}_z,\tilde{z})$ lies strictly above the line through the two points~$(1/p_z,z)$ and $(1/p,\overline{s}_p)$ in the DeVore-Triebel diagram. Moreover, by complex interpolation it becomes obvious that the set of parameters
\begin{align*}%\label{def:set}
	\left\{ \left(\frac{1}{\rho},s\right)\in [0,\infty)^2 \sep S(\Omega)\subseteq B^{s}_{\rho,\rho}(\Omega) \right\}
\end{align*}
is necessarily convex and that each $(1/\rho,\overline{s}_\rho)$ with $0<\rho\leq \infty$ belongs to its boundary.

\item For $1<p<\infty$, the regularity of a function in the scale~\eqref{eq:scale} is intimately related to the convergence rate of the best $N$-term approximation, if the error is measured in~$L_p(\Omega)$.
However, if the error is to be measured in the norm of some other Sobolev space~$W^r_{p}(\Omega)$ with $r>0$ (describing, for instance, the energy space), then the scale changes to
\begin{equation*}
	B^\alpha_{\tau,\tau}(\Omega), \qquad \frac{1}{\tau}=\frac{\alpha-r}{d}+\frac{1}{p}, \quad \alpha > r.
\end{equation*}
Since this is just a shift of the $L_p(\Omega)$-Sobolev embedding line, our analysis carries over to this case mutatis mutandis. For the ease of presentation we omit the details. Moreover, we can replace the underlying Lipschitz domain $\Omega$ by a (patchwise smooth) manifold; cf.\ \cite{DahHar+2018, DahWei2015, Wei2016}.
\end{enumerate}
\end{remark}

% Now we are ready to give a detailed proof of~\autoref{thm:main:new}.
 We close this section with a detailed proof of~\autoref{thm:main:new}. 

\begin{proof}[{Proof of~\autoref{thm:main:new}}]
Relation~\eqref{eq:z:s:alpha} follows by contradiction due to the fact that for all $0<p_1<p_0<\infty$ and $s_1<s_0$ there holds $B^{s_0}_{p_0,p_0}(\Omega)\hookrightarrow B^{s_1}_{p_1,p_1}(\Omega)$, see~\autoref{prop:embeddings}\ref{it:embedd:Sob:eps}. This embedding also implies that $\overline{\alpha}_p=\infty$ if $\overline{s}_p=\infty$. Thus, we are left with proving~\eqref{eq:alpha:bound} for $\overline{s}_p<\infty$.
Again we argue by contradiction.
Assume $S(\Omega)\subseteq B^\alpha_{\tau,\tau}(\Omega)$, $1/\tau=\alpha/d+1/p$, for some $\alpha>\overline{s}_p\cdot (\overline{s}_p-\mu )/(z-\mu )$. 
Since $%A^z_{p_z,q_z}(\Omega)\hookrightarrow 
S(\Omega)\subseteq B^{\tilde{z}}_{p_z,p_z}(\Omega)$ for all $\tilde{z}<z$, 
%see Proposition~\ref{prop:embeddings}\ref{it:embedd:Sob:eps}, 
we also know that $S(\Omega)\subseteq B^{\tilde{s}}_{\tilde{p},\tilde{p}}(\Omega)$ with $\tilde{s}=(1-\theta)\,\tilde{z}+\theta\, \alpha$ and $1/\tilde{p}=(1-\theta)/p_z+\theta/\tau$ for all $\theta\in (0,1)$, see~\autoref{prop:interpol}.
In particular, if we choose 
\begin{equation*}
	\theta
	=\theta_0
	:=\frac{1/p-1/p_z}{1/\tau-1/p_z}
	=\frac{1/p-1/p_z}{\alpha/d+1/p-1/p_z}
	=\frac{\overline{s}_p-\mu }{\alpha+\overline{s}_p-\mu}
\in (0,1),
\end{equation*}
we obtain $S(\Omega)\subseteq B^{\tilde{s}}_{p,p}(\Omega)$ for all $\tilde{s}=(1-\theta_0)\,\tilde{z}+\theta_0\,\alpha$ with $\tilde{z}<z$. 
Since $\alpha>\overline{s}_p\cdot(\overline{s}_p-\mu )/(z-\mu )$, we have 
\begin{align*}
	(1-\theta_0)\,z+\theta_0\,\alpha
	=\frac{\alpha\,(z+\overline{s}_p-\mu )}{\alpha+\overline{s}_p-\mu }
	=\frac{z+\overline{s}_p-\mu }{1+(\overline{s}_p-\mu)/\alpha}
	>\frac{z+\overline{s}_p-\mu }{(z+\overline{s}_p-\mu)/\overline{s}_p}
	=\overline{s}_p.
\end{align*}
Therefore, there exists $\tilde{z}<z$, such that $s:=(1-\theta_0)\,\tilde{z}+\theta_0\,\alpha>\overline{s}_p$, which means that $S(\Omega)\subseteq B^{s}_{p,p}(\Omega)$ for some $s>\overline{s}_p$. But this contradicts the maximality of $\overline{s}_p$.
\end{proof}

\section{Examples}\label{sect:example}
In this section we apply \autoref{thm:main:new} to three sample problems: the Poisson equation, the $p$-Poisson problem, and the inhomogeneous stationary Stokes equation.

\subsection{The Poisson equation}\label{sect:Poisson}
Let us consider the Poisson equation with zero Dirichlet boundary conditions 
%inhomogeneous Dirichlet problem
\begin{equation}\label{eq:Poisson}
	\left.
	\begin{alignedat}{3}
	\Delta u &= f \quad \text{on } \Omega,	\\
	u &=0  \quad\, \text{on } \partial\Omega,\;	\\
	\end{alignedat}
	\right\}	
\end{equation}
on a bounded Lipschitz domain $\Omega\subseteq \R^d$, $d\geq 2$.  Points where the boundary $\partial\Omega$ of the underlying domain $\Omega$ is not smooth  are known to have negative effects on the regularity of the solution~$u$ to~\eqref{eq:Poisson}.
While on smooth domains we have the usual shift 
\begin{equation*}
	f\in W^{s-2}_p(\Omega) \quad \Longrightarrow \quad u\in W^{s}_p(\Omega),
\end{equation*}
this mechanism fails if we allow the boundary of $\Omega$ to be merely $\cont^1$. In this case, for instance, $f\in W^{-1/2}_2(\Omega)$ does not necessarily imply $u\in W^{3/2}_2(\Omega)$.
This problem has been intensively studied in~\cite{JK95} by Jerison and Kenig; see also~\cite{FabMenMit1998,May2005}. Therein one may find a precise description of the range of parameters $(1/p,s)$ that allow for shift theorems for Equation~\eqref{eq:Poisson} in Bessel potential spaces and in Besov spaces.
The sharpness of this range is underpinned by several counterexamples, see, in particular,~\cite[Section~6]{JK95}.
Motivated by these results and by the relevance of the regularity in Sobolev spaces and in the scales~\eqref{eq:scale} of Besov spaces in (non-)linear approximation theory, Dahlke and DeVore~\cite{DahDeV1997} analyzed the regularity of the Poisson equation in Besov spaces with integrability parameter less than one.
Put together, the positive results from~\cite{JK95} and \cite{DahDeV1997} guarantee the following: If we are only interested in the  consequences of the lack of boundary smoothness 
%boundary singularities 
and therefore assume that $f\in\cont^\infty(\overline{\Omega})$, then the solution $u\in W^1_{2,0}(\Omega)$ to the corresponding equation~\eqref{eq:Poisson} is contained in every Besov space~$B^r_{q,q}(\Omega)$ represented by a point $(1/q,r)$ within the shaded area in the DeVore-Triebel diagram in~\autoref{fig:DeVoreTriebel02}.
Using the terminology from the previous sections, we set
\begin{align}\label{eq:Poisson:S}
	S(\Omega):=\ggklam{u\in W^1_{2,0}(\Omega) \sep \Delta u\in \cont^\infty(\overline{\Omega})}.
\end{align}
Then 
\begin{align*}%\label{eq:JK:DahDeV:S}
	S(\Omega)\subseteq B^r_{q,q}(\Omega) \qquad\text{for all}\quad
	0<r<1+\frac{1}{q} \quad\text{and}\quad
	0<\frac{1}{q}<\frac{d+1}{d-1}
\end{align*}
such that, in particular,
\begin{align}\label{eq:JK:DahDeV:bar}
	\overline{s}_p(S(\Omega))\geq 1+\frac{1}{p}
	\qquad\text{and}\qquad
	\overline{\alpha}_p(S(\Omega))\geq \ssgrklam{1+\frac{1}{p}}\frac{d}{d-1}
\end{align}
for every $1<p<\infty$. 
The following theorem asserts the existence of bounded $\cont^1$ domains on which these lower bounds for $\overline{s}_p$ and $\overline{\alpha}_p$ become also upper bounds.

\begin{theorem}\label{thm:Poisson:bounds}
	For $d\geq 2$, there exists a bounded $\cont^1$ domain $\Omega\subseteq\R^d$ such that if $S(\Omega)$ is defined as in~\eqref{eq:Poisson:S}, then for arbitrary $1<p<\infty$ there holds
	\begin{align*}%\label{eq:JK:neg:bar}
		\overline{s}_p(S(\Omega))= 1+\frac{1}{p}
		\qquad\text{and}\qquad
		\overline{\alpha}_p(S(\Omega))= \ssgrklam{1+\frac{1}{p}}\frac{d}{d-1}.
	\end{align*}
\end{theorem}

Our proof of~\autoref{thm:Poisson:bounds} below is based on a counterexample by Jerison and Kenig of a~$\cont^1$ domain $\Omega\subseteq\R^d$, $d\geq 2$, for which there exists a function $f\in\cont^\infty(\overline{\Omega})$, such that the second derivatives of the solution $u\in W^1_{2,0}(\Omega)$ to the corresponding equation~\eqref{eq:Poisson} are not contained in $L_1(\Omega)$, thus $u\notin W^2_1(\Omega)$. We refer to~\cite[Theorem~1.2(b)]{JK95} for the statement and to \cite[Section~6]{JK95} for the corresponding counterexample.
For such a solution %$u\notin W^2_1(\Omega)$ 
to~\eqref{eq:Poisson} we prove the following.

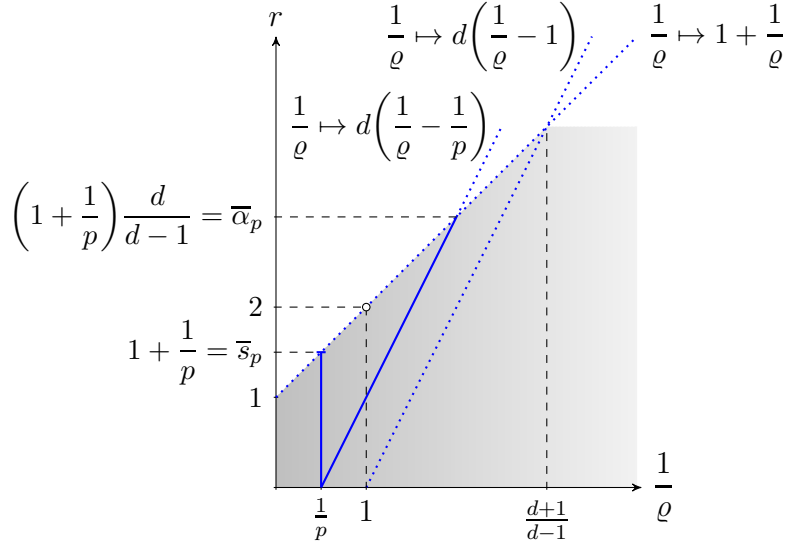
\begin{figure}
\begin{center}
\begin{tikzpicture}[scale=.6]

% The coordinate axes
\shade [left color=gray!50, right color=gray!10, shading angle = 90] (0,0) -- (0,2) -- (6,8) -- (8,8) -- (8,0);
\draw [->,>=stealth'] (0,-0.05) -- (0,10) node[above] {$r$};
\draw [->,>=stealth'] (-0.05,0) -- (8.1,0) node[right] {$\displaystyle\frac{1}{\rho}$};

%\draw [blue, thick, dotted] (4,6)--(6,6);
\draw [blue, thick, dotted] (2,0)--(7,10) node[left] {\color{black}\small $\displaystyle\frac{1}{\rho}\mapsto d\ssgrklam{\frac{1}{\rho}-1}$};

\draw (-0.05,2)--(0.05,2);
\draw (-0.05,2)  node[left] {\small $1$};
\draw [dashed] (6,-0.05)--(6,8);
\draw (6,-0.05) node[below] {$\frac{d+1}{d-1}$};

\draw [blue, dotted, thick] (4,6) -- (5,8);
\draw [blue, thick] (1,0) -- (4,6);
\draw [blue, dotted, thick] (0,2)--(8,10);
\draw (5,8) node[left] {{\small $\displaystyle\frac{1}{\rho}\mapsto d\ssgrklam{\frac{1}{\rho}-\frac{1}{p}}$}} ;
\draw (8,10) node[right] {{\small $\displaystyle\frac{1}{\rho}\mapsto 1+\frac{1}{\rho}$}} ;
\draw [blue,thick](1,0)--(1,3);
%\draw [blue,thick, dotted](1,3)--(1,8);
\draw [blue,thick] (0.9,3)--(1.1,3);
\draw (1,-0.05)--(1,0) node[below] {$\frac{1}{p}$};
\draw [dashed] (4,6)--(-0.05,6) node[left] {{\small $\displaystyle \ssgrklam{1+\frac{1}{p}}\frac{d}{d-1}=\overline{\alpha}_p$}};
\draw [dashed] (-0.05,3)--(1,3);
\draw (-0.05,3) node[left] {{\small $\displaystyle 1+\frac{1}{p}=\overline{s}_p$}};
\draw [dashed] (2,-0.05)--(2,4);
\draw (2,-0.05) node[below] { $1$};
\draw [dashed] (-0.05,4) -- (2,4);
\draw (-0.05,4) node[left] {\small $2$};

\draw (2,4) node[circle,draw=black, fill=white, inner sep=1pt]  {};

\end{tikzpicture}
\caption[DeVoreTriebel02]{\begin{tabular}[t]{l}Visualization of the Besov regularity of the Poisson equation with \\ smooth right-hand side on bounded $\cont^1$ domains in a DeVore-Triebel \\ diagram.\end{tabular}}\label{fig:DeVoreTriebel02}
\end{center}
\end{figure}

\begin{lemma}\label{lem:JK:neg}
	Let $d\geq 2$. 
	Moreover, let $\Omega\subseteq\R^d$ be a $\cont^1$ domain  for which there exists a function $f\in\cont^\infty(\overline{\Omega})$ such that the unique solution $u\in W^1_{2,0}(\Omega)$ to the corresponding Poisson equation~\eqref{eq:Poisson} satisfies $u\notin W^2_1(\Omega)$. %(such domains and functions exist due to~\cite[Theorem~1.2(b)]{JK95}).
	Then the following statements hold.
	\begin{enumerate}[align=right,label=\textup{(\roman*)}] 
		\item\label{lem:JK:neg:2} $u\notin B^2_{1,1}(\Omega)$.
		\item\label{lem:JK:neg:3} If $1<p<\infty$ and $\displaystyle s>1+\frac{1}{p}$, then $u\notin B^s_{p,p}(\Omega)$.
		\item\label{lem:JK:neg:4} $u\in F^{1+1/p}_{p,2}(\Omega)$ for all $2\leq p<\infty$.
		\item\label{lem:JK:neg:5} Let $1<p<\infty$ and let $0<\tau,\alpha<\infty$ be such that $\displaystyle \frac{1}{\tau}=\frac{\alpha}{d}+\frac{1}{p}$. Moreover, assume that
\[
\alpha>\tilde{\alpha}_p:=\ssgrklam{1+\frac{1}{p}}\frac{d}{d-1}
\qquad
\text{or}
\qquad
\alpha=\tilde{\alpha}_p\text{ and }\tau<1.
\]
%\begin{itemize}
%\item $\displaystyle\alpha>\tilde{\alpha}_p:=\ssgrklam{1+\frac{1}{p}}\frac{d}{d-1}$
%
%or
%
%\item $\alpha=\tilde{\alpha}_p$ and $\tau<1$.
%\end{itemize}		
Then $u\notin B^\alpha_{\tau,\tau}(\Omega)$.
%		\item\label{lem:JK:neg:5} If $1<p<\infty$ and $\displaystyle\alpha>\tilde{\alpha}_p:=\ssgrklam{1+\frac{1}{p}}\frac{d}{d-1}$, then $u\notin B^\alpha_{\tau,\tau}(\Omega)$, $\displaystyle \frac{1}{\tau}=\frac{\alpha}{d}+\frac{1}{p}$.
%		 Moreover, $u\notin B^{\tilde{\alpha}_p}_{\tau,\tau}(\domain)$ if $\displaystyle\frac{1}{\tau}=\frac{\tilde{\alpha}_p}{d}+\frac{1}{p}>1$.
	\end{enumerate}
\end{lemma}

\begin{proof} We prove the four statements successively.
\begin{enumerate}
\item[\ref{lem:JK:neg:2}.] The assertion $u\in B^2_{1,1}(\Omega)$ would contradict our assumption that $u\notin W^2_1(\Omega)$ since $B^2_{1,1}(\Omega)\hookrightarrow W^2_1(\Omega)$, which follows, e.g., from~\cite[Theorem~2.3.8(i) \& Proposition~2.5.7(i)]{T83}.

\item[\ref{lem:JK:neg:3}.] Suppose that $u\in B^s_{p,p}(\Omega)$ for some $1<p<\infty$ and $s>1+1/p$. W.l.o.g.\ we may also assume that $s<2$. From \cite[Theorem~4.1]{DahDeV1997} we can deduce that $u\in B^{r}_{q,q}(\Omega)$ with $1/q=1+\varepsilon$ and $r=2+\varepsilon\, (2-s)/(1- 1/p )$ for all $0<\varepsilon<2/(d-1)$.
Then by~\autoref{prop:interpol} we have
\begin{equation*}
	u\in 
	\geklam{B^s_{p,p}(\Omega),B^r_{q,q}(\Omega)}_\theta = B^2_{1,1}(\Omega)
	\qquad\text{for}\qquad
	\theta=\frac{1-1/p}{1-1/p+\varepsilon}\in (0,1).
\end{equation*}
However, this contradicts~\ref{lem:JK:neg:2}.

\item[\ref{lem:JK:neg:4}.] We prove this assertion with an argument used in~\cite[point 4. on page~2167]{Cos2019}: Let us extend $f$ to the whole of $\R^d$ such that the extension (also denoted by $f$) is at least smooth enough to be contained in $F^{-1+1/p+\varepsilon}_{p,2}(\R^d)$ for some $\varepsilon>0$. Then the equation $\Delta v=f$ on $\R^d$ has a unique solution $v\in F^{1+1/p+\varepsilon}_{p,2}(\R^d)$ and $v|_{\partial\Omega}\in B^{1+\varepsilon}_{p,p}(\partial\Omega)\hookrightarrow F^1_{p,2}(\partial\Omega)$. Therefore, $\tilde{u}:=v-u$ is a harmonic function on $\Omega$ with trace $\tilde{u}|_{\partial\Omega}\in F^1_{p,2}(\partial\Omega)$. From \cite[Theorem~5.15(b)]{JK95} it thus follows that $\tilde{u}\in F^{1+1/p}_{p,2}(\Omega)$ and hence also $u=\tilde{u}-v\in F^{1+1/p}_{p,2}(\Omega)$.

\item[\ref{lem:JK:neg:5}.]  We first consider the case $\alpha>\tilde{\alpha}_p$.  Theorems 1.1 and 1.3 of  \cite{JK95} together with part~\ref{lem:JK:neg:3} imply that $\overline{s}_p:=\overline{s}_p(\{u\})=1+1/p$ for all $1<p<\infty$. Now fix $1<p<p_z<\infty$. Then, we may apply~\autoref{thm:main:new} with 
$z:=\overline{s}_{p_z}=1+1/p_z$ and
\begin{align*}
	\mu 
	= \overline{s}_p-d\,\ssgrklam{\frac{1}{p}-\frac{1}{p_z}}
	= 1+\frac{1}{p}-d\,\ssgrklam{\frac{1}{p}-\frac{1}{p_z}}
	= 1+\frac{1}{p_z}-(d-1)\ssgrklam{\frac{1}{p}-\frac{1}{p_z}}
	< 1+\frac{1}{p_z},
\end{align*}
to obtain
\begin{align*}
	\overline{\alpha}_p(\{u\})
	\leq \overline{s}_p\cdot\frac{\overline{s}_p-\mu}{\overline{s}_{p_z}-\mu}
	= \ssgrklam{1+\frac{1}{p}}\frac{d}{d-1}
	=\tilde{\alpha}_p
\end{align*}
which obviously proves~\ref{lem:JK:neg:5} if $\alpha>\tilde{\alpha}_p$.

 The fact that  $u\notin B^{\tilde{\alpha}_p}_{ \tau,\tau}(\Omega)$,  $1/\tau=\tilde{\alpha}_p/d + 1/p$,  if $\tau<1$ follows from parts~\ref{lem:JK:neg:2} and~\ref{lem:JK:neg:4} by another complex interpolation argument: Since $u \in F^{3/2}_{2,2}(\Omega)=B^{3/2}_{2,2}(\Omega)$ and the points $(1/2,3/2)$, $(1,2)$, and $(\tilde{\alpha}_p/d+1/p, \tilde{\alpha}_p)$ lie on the same line of slope $1$ through $(0,1)$ in a DeVore-Triebel diagram, the statement $u\in B^{\tilde{\alpha}_p}_{\tau,\tau}(\Omega)$ would contradict~\ref{lem:JK:neg:2}.\qedhere 
\end{enumerate}
\end{proof}

\begin{proof}[Proof of \autoref{thm:Poisson:bounds}]
Due to Jerison and Kenig~\cite[Theorem~1.2(b)]{JK95}, there exist $\Omega\subseteq\R^d$ and $f\in \cont^\infty(\overline{\Omega})$, such that the assumptions of~\autoref{lem:JK:neg} are satisfied. Therefore, the assertion follows from~\autoref{lem:JK:neg} and~\eqref{eq:JK:DahDeV:bar}. 
\end{proof}

We conclude this subsection with some further remarks.
\begin{remark}\label{rem:grisvard}
It is worth mentioning that the bounds in~\autoref{thm:Poisson:bounds} are due to worst-case scenarios regarding the behaviour of $\cont^1$ boundaries. However, for large classes of domains, which are not even necessarily of class $\cont^1$, the regularity indices $\overline{s}_p(S(\Omega))$ and~$\overline{\alpha}_p(S(\Omega))$ with~$S(\Omega)$ as defined in~\eqref{eq:Poisson:S} may be higher, at least for certain $1<p<\infty$.
For instance, if $\Omega\subseteq\R^2$ is a polygonal domain with maximal interior angle $\kappa_0\in (\pi,2\pi)$, then Grisvard~\cite{Gri1985,Gri1992} shows that 
\begin{equation}\label{eq:Poisson:bars:Pol}
	\overline{s}_p(S(\Omega))=\frac{2}{p}+\frac{\pi}{\kappa_0}, \qquad 1<p<\infty,
\end{equation}
which is strictly greater than $1+1/p$ whenever $p<\kappa_0/(\kappa_0-\pi)$. Moreover, it is known from~\cite{Dah1999} that
\begin{equation*}
	\overline{\alpha}_2(S(\Omega))=\infty.
\end{equation*}
Note that this does not contradict~\autoref{thm:main:new} since~\eqref{eq:Poisson:bars:Pol} implies that for any fixed $1<p<\infty$ and all $p_z>p$, there is no $z>\mu(p_z,p,\overline{s}_p,2)$ such that $S(\Omega)\subseteq B^s_{p_z,p_z}(\Omega)$ for all $s<z$. 
\end{remark}

\begin{remark}
In~\cite{Cos2019} Costabel constructs bounded $\cont^1$ domains $\Omega\subseteq\R^d$ of arbitrary dimension~$d\geq 2$, for which there exists $f\in\cont^\infty(\overline{\Omega})$ such that the solution $u$ to the corresponding Poisson equation~\eqref{eq:Poisson} is contained in $W^{3/2}_2(\Omega)$, but not in $W^{1+1/p+\varepsilon}_p(\Omega)$ for any $1\leq p<\infty$ and any $\varepsilon>0$; see, in particular, Theorem~1.2 and Remark~1.3 therein.
\autoref{lem:JK:neg} above shows that the counterexample provided by Jerison and Kenig in~\cite[Section~6]{JK95} as a proof of Theorem~1.2(b) therein has these properties, too.
\end{remark}

\subsection{The \texorpdfstring{$p$}{p}-Poisson problem}
Our second example 
%can be viewed as a nonlinear generalization of the classical Poisson problem~\eqref{eq:Poisson}.  To this end, 
is the $p$-Poisson problem for some fixed $1<p<\infty$.
For $d\geq 2$, let again $\Omega \subseteq\R^d$ denote a bounded Lipschitz domain.
% and assume that $1<p<\infty$ is fixed. 
Given $f\in W^{-1}_{p'}(\Omega)$ with $1/p+1/p'=1$, we seek the unique weak solution $u\in W_{p,0}^1(\Omega)$ to 
\begin{align}\label{eq:p-poisson}
	\left.
	\begin{alignedat}{3}
		\Delta_p u &= f \quad \!\text{on}\quad \Omega, \\
		u&=0 
		\quad \text{on}\quad \partial\Omega,\\
	\end{alignedat}
	\right\}	
\end{align}
where $\Delta_p u := \mathrm{div}(\abs{\nabla u}_2^{p-2} \, \nabla u)$ denotes the $p$-Laplace operator. 

For this problem various local and global regularity results are known; we refer, e.g., to~\cite{BalDie+2019, DahDieHar+2014,Ebm2002,HarWei2018,Sav1998} and the references therein. Our subsequent analysis relies on the following result.

\begin{prop}[{Ebmeyer~\cite[Theorem~2.4]{Ebm2002}}]\label{prop:p-poisson-reg}
	For $d\geq 2$ let $\Omega \subseteq\R^d$ denote a bounded polyhedral Lipschitz domain. Moreover, let $1<p\leq 2$ and $f\in L_{p'}(\Omega)$. Then the unique weak solution to \eqref{eq:p-poisson} satisfies
	\begin{align}\label{eq:Zexample}
		u \in W_{p_z}^{s}(\Omega)
		\qquad\text{for all}\quad s<\frac{3}{2} \quad \text{and}\quad  p_z:=p_z(d,p):=\frac{p}{1-(2-p)/(2d)}.
	\end{align}
\end{prop}

Although, to the best of our knowledge, even in this restricted setting the exact value of $\overline{s}_p$ is unknown, we can apply our main~\autoref{thm:main:new} in order to deduce the following statement:
\begin{theorem}\label{thm:pPoisson}
	For $d\geq 2$ let $\Omega \subseteq\R^d$ denote some bounded polyhedral Lipschitz domain. 
	Given $1<p<2$ let $S(\Omega)$ denote the set of solutions to the $p$-Poisson problem \eqref{eq:p-poisson} with right-hand sides $f\in L_{p'}(\Omega)$. 
	Then for the regularity indices $\overline{s}_p$ and $\overline{\alpha}_p$ as defined in \eqref{eq:def:sbar} and~\eqref{eq:def:alphabar}, respectively, one of the following cases applies:
	\begin{enumerate}[align=right, label=\textup{\arabic*.)}]
		\item  $3/2 \leq \overline{s}_p < 1+1/p$ and
		\begin{equation*}
		\overline{s}_p \leq \overline{\alpha}_p \leq \overline{s}_p \, \frac{1+1/p-3/2}{1+1/p-\overline{s}_p}.
		\end{equation*}
		\item $1+1/p \leq \overline{s}_p \leq \overline{\alpha}_p$.
	\end{enumerate}
\end{theorem}
\begin{proof}
	For $1<p<2$ the parameter $p_z$ in \eqref{eq:Zexample} is strictly larger than $p$. Using that $W^s_{p_z}(\Omega)=B^s_{p_z,p_z}(\Omega)$ for $0<s\notin \N$, we thus can apply \autoref{thm:main:new} with this $p_z$ and $z:=3/2$. This yields that in any case there holds
	\begin{equation*}
		\frac{3}{2} \leq \overline{s}_p \leq \overline{\alpha}_p.
	\end{equation*}
	Moreover, $\mu=\mu(p_z,p,\overline{s}_p,d)=\overline{s}_p-1/p+1/2$ is strictly less than $z=3/2$ if, and only if, $\overline{s}_p<1+1/p$. In this case, also Formula~\eqref{eq:alpha:bound} in \autoref{thm:main:new} applies which proves the upper bound on $\overline{\alpha}_p$ in case 1.). Hence, the proof is complete.
\end{proof}

Let us add some remarks also for this example.
\begin{remark}
There exist statements similar to \autoref{prop:p-poisson-reg} also for $p \geq 2$; see, e.g., Ebmeyer~\cite{Ebm2002} for details. However, in this case the analogue of \eqref{eq:Zexample} does not provide \emph{additional} information; cf.\ \autoref{rem:main}\ref{rem:main:fail}. That is, using \autoref{thm:main:new} not much can be said except that $\overline{\alpha}_p(S(\Omega)) \geq \overline{s}_p(S(\Omega))$ might be unbounded. Anyway, again this agrees well with results due to Dahlke~\cite{Dah1999}, %~\cite[Theorem~2.1]{Dah1999} 
who showed that for $p=d=2$ and smooth right-hand sides we indeed have $\overline{\alpha}_2(S(\Omega))=\infty > \overline{s}_2(S(\Omega))$; see also \autoref{rem:grisvard} above.
\end{remark}

\begin{remark}
\autoref{thm:pPoisson} shows that on polyhedral Lipschitz domains the maximal $L_p(\Omega)$-Sobolev smoothness~$\overline{s}_p$ is at least $3/2$. In~\cite[Theorem~2']{Sav1998} Savar{\'e} proved that this remains true on general Lipschitz domains under the weaker condition that $f\in W^{-1/2}_{p'}(\Omega)$. Moreover, in~\cite[Remark~4.3]{Sav1998} he even claims optimality. 
However, if we stick to the stronger assumptions that $\Omega$ is polyhedral Lipschitz and $f\in L_{p'}(\Omega)$, we may use positive Besov regularity results w.r.t.\ the scale \eqref{eq:scale} in order to conclude a better lower bound. Indeed, combining \autoref{prop:p-poisson-reg} with \autoref{rem:main}\ref{rem:main:lower} shows that
\begin{equation*}
	\overline{s}_p 
	\geq \tilde{s}_p \geq \alpha \cdot\frac{z+d\,(1/p-1/p_z)}{\alpha+d\,(1/p-1/p_z)} 
	= \left( 1+\frac{1}{p} \right) \, \frac{\alpha}{\alpha+1/p-1/2} =:\widehat{s}_p(\alpha)
	\qquad \text{for all} \quad \alpha \leq \overline{\alpha}_p.
\end{equation*}
Note that this lower bound is strictly monotonically increasing in $\alpha$, where
\begin{equation*}
	\frac{3}{2} = \widehat{s}_p(3/2) < \widehat{s}_p(\alpha) < 1+ \frac{1}{p}, \qquad \alpha > \frac{3}{2}.
\end{equation*}
Results of Dahlke et~al.~\cite[Theorem~4.20]{DahDieHar+2014} imply that on bounded polygonal domains $\Omega\subseteq\R^2$,
\begin{equation}\label{eq:DahlkepPoisson}
	S(\Omega):=\left\{u\in W^1_{p,0}(\Omega) \sep \Delta_p u \in L_\infty(\Omega) \right\} \subseteq B^{\alpha}_{\tau,\tau}(\Omega), 
	\quad \frac{1}{\tau}=\frac{\alpha}{2}+\frac{1}{p}, \quad \text{for all} \quad 0<\alpha < 2,
\end{equation}
such that in this case
\begin{equation*}
	\overline{s}_p(S(\Omega)) \geq 2\, \frac{1+1/p}{1+1/p+1/2}, \qquad 1<p \leq 2.
\end{equation*}
Furthermore, recent results indicate that we may replace $L_\infty(\Omega)$ by $L_{p'}(\Omega)$ in \eqref{eq:DahlkepPoisson}.
\end{remark}

\subsection{The inhomogeneous stationary Stokes problem}
Our third and final example is the inhomogeneous stationary Stokes system
\begin{align}\label{def:StokesProb}
	\left.
	\begin{alignedat}{3}
		-\Delta u + \nabla \pi &= f \quad \text{in} \quad \Omega, \\
		\mathrm{div}(u) &= g \quad \,\text{in} \quad \Omega, \\
		u_{|_{\partial\Omega}} &= h \quad \text{on} \quad \partial\Omega,\; \\
	\end{alignedat}
	\right\}	
\end{align}
where $\Omega\subseteq\R^d$ is again a bounded Lipschitz domain ($d\geq 2$) and
$f$, $g$, and $h$ are given functions (or distributions) on $\Omega$ and $\partial\Omega$, respectively, such that the compatibility condition
\begin{align}\label{def:Compatibility}
\int_{\partial\Omega} h(y) \cdot \eta(y) \d y = \int_{\Omega} g(x) \d x
\end{align}
is satisfied; here, $\eta$ denotes the outward unit normal vector to $\partial\Omega$.

For this problem, Mitrea and Wright \cite{MitWri2012} showed that a suitably modified regularity shift holds in a range of parameters $\mathcal{R}_{d,\epsilon}\subseteq \R\times (0,\infty]$ similar to the one established by Jerison and Kenig~\cite{JK95} for the classical Poisson problem; see \cite[page~178]{MitWri2012} for a precise definition of $\mathcal{R}_{d,\varepsilon}$. Without going into details, this range depends on a ``roughness parameter'' $\epsilon=\epsilon(\Omega)\in(0,1]$ which measures the Lipschitz nature of $\Omega$. However, for sufficiently smooth domains, e.g., when $\partial\Omega\in \cont^1$, we may take $\epsilon=1$.

\begin{prop}[{Mitrea and Wright \cite[Theorem~1.5/10.15]{MitWri2012}}]\label{prop:Mitrea}
	For $d\geq 2$ let $\Omega\subseteq\R^d$ be a bounded Lipschitz domain. 
	Moreover, let $A\in\{B,F\}$, as well as $(d-1)/d<p \leq \infty$, $0<q\leq \infty$, and $(d-1)\max\{1/p-1,0\} < s < 1$ with $(s,p)\in \mathcal{R}_{d,\epsilon(\Omega)}$, where $\max\{p,q\}<\infty$ if $A=F$. 
	Then for 
	\begin{equation*}
		f\in A^{s+1/p-2}_{p,q}(\Omega)^d, \quad g\in A^{s+1/p-1}_{p,q}(\Omega), 
		\quad \text{and}\quad 
		h\in \begin{cases}
			B^{s}_{p,q}(\partial\Omega)^d & \text{if } A=B,\\
			F^{s}_{p,p}(\partial\Omega)^d & \text{if } A=F,
		\end{cases}
	\end{equation*}
	there exists a solution $(u,\pi)\in A^{s+1/p}_{p,q}(\Omega)^d\times A^{s+1/p-1}_{p,q}(\Omega)$ to the inhomogeneous stationary Stokes system \eqref{def:StokesProb}, \eqref{def:Compatibility}.
	Moreover, it is unique modulo the addition of locally constant functions in $\Omega$ to the pressure $\pi$.
\end{prop}

This statement can be used to conclude the subsequent regularity assertion which provides all necessary information for the application of \autoref{thm:main:new} to the Stokes problem.
\begin{lemma}\label{lem:Stokes}
	For $d\geq 2$ let $\Omega\subseteq\R^d$ denote a bounded Lipschitz domain with roughness parameter $\epsilon(\Omega)\in(0,1]$. 
	Further, let $0<s<1$, as well as $\sigma:=\min_{j\in\{1,2,3\}}\limits \sigma_j \geq 0$ and  
	\begin{align*}%\label{cond:data}
		f\in H^{s-3/2+\sigma_1}(\Omega)^d, \quad g\in H^{s-1/2+\sigma_2}(\Omega), \quad \text{and} \quad h\in H^{s+\sigma_3}(\partial\Omega)^d.
	\end{align*}
	Then solutions $(u,\pi)$ to \eqref{def:StokesProb}, \eqref{def:Compatibility} exist and satisfy $(u,\pi) \in H_{p}^{s+1/p}(\Omega)^d\times H_{p}^{s+1/p-1}(\Omega)$ for all $p \in [2,\infty)$ with
	\begin{equation}\label{cond:p}
		\frac{1}{2} - \min\left\{ \frac{\epsilon(\Omega)}{2}, \frac{\sigma}{d-1}\right\} 
		\leq \frac{1}{p} 
		\leq \frac{1}{2}.
	\end{equation}
\end{lemma}
\begin{proof}
Due to simple embeddings we may w.l.o.g.\ assume that 
%$0 \leq \sigma <1/2$;
 $0 \leq \sigma \leq (d-1)/2$;  see \autoref{prop:embeddings}\ref{it:embedd:Sob:eps}. Further let $s\in\R$ and $p\in[2,\infty)$. Then, according to \autoref{def:spaces} and \autoref{prop:embeddings}, there holds
\begin{equation*}
	H^{s-3/2+\sigma_1}(\Omega) 
	\hookrightarrow F_{2,2}^{s-3/2+\sigma}(\Omega) 
	\hookrightarrow F_{p,2}^{s_1}(\Omega)
	\hookrightarrow F_{p,2}^{s+1/p-2}(\Omega)
\end{equation*}
provided that
\begin{align*}
	s_1 &:=s+\frac{1}{p}-2+\sigma + (d-1) \left( \frac{1}{p}-\frac{1}{2}\right) %\\
	%&
	\geq s+\frac{1}{p}-2.
\end{align*}
Note that this inequality is satisfied if $p$ is chosen such that
\begin{align}\label{cond:psigma}
	\frac{1}{2} - \frac{\sigma}{d-1} \leq \frac{1}{p}.
\end{align}
Moreover, similar calculations show that the same condition~\eqref{cond:psigma} implies the embeddings $H^{s-1/2+\sigma_2}(\Omega)\hookrightarrow F_{p,2}^{s+1/p-1}(\Omega)$ and $H^{s+\sigma_3}(\partial\Omega)\hookrightarrow F_{p,p}^{s}(\partial\Omega)$.
Hence, our assumptions %\eqref{cond:data} 
on the data give
\begin{align*}%\label{cond:data2}
	f\in F^{s+1/p-2}_{p,2}(\Omega)^d, \quad g\in F^{s+1/p-1}_{p,2}(\Omega), \quad \text{and} \quad h\in F^{s}_{p,p}(\partial\Omega)^d
\end{align*}
with $0<s<1$ and each $p\in[2,\infty)$ with~\eqref{cond:psigma}.
Furthermore, it can be checked easily that $(s,p)\in \mathcal{R}_{d,\epsilon(\Omega)}$ whenever $0<s<1$ and $p\in[2,\infty)$ with 
\begin{equation*}
	\frac{1}{2} - \frac{\epsilon(\Omega)}{2} \leq \frac{1}{p}.
\end{equation*}
Thus, the claim follows from \autoref{prop:Mitrea} applied for $A:=F$, $q:=2$, as well as $0<s<1$ and $p\in[2,\infty)$ restricted by~\eqref{cond:p}, and \autoref{def:spaces}.
\end{proof}

\begin{theorem}\label{thm:stokes}
	For $d\geq 2$ let $\Omega\subseteq\R^d$ denote a bounded Lipschitz domain with roughness parameter $\epsilon=\epsilon(\Omega)\in (0,1]$. 
	Let $S_u(\Omega)$ and $S_\pi(\Omega)$ denote the sets of solutions $(u,\pi)$ to the inhomogeneous stationary Stokes problem \eqref{def:StokesProb}, \eqref{def:Compatibility} with
	\begin{align*}
		f\in H^{-1/2+\sigma_1}(\Omega)^d, \quad g\in H^{1/2+\sigma_2}(\Omega), \quad \text{and} \quad h\in H^{1+\sigma_3}(\partial\Omega)^d,
	\end{align*}
	where
	\begin{equation*}
		\sigma:=\min_{j\in\{1,2,3\}}\limits \sigma_j>0.
	\end{equation*}
	Moreover, let $m:=\min\left\{ (d-1)\, \epsilon/2,\sigma\right\}$. 
	Then for the regularity indices $\overline{s}_2:=\overline{s}_2(S_u(\Omega))$ and $\overline{\alpha}_2:=\overline{\alpha}_2(S_u(\Omega))$ of (each component of) the velocity $u$ one of the following cases applies:
\begin{enumerate}[align=right,label=\textup{\arabic*.)}] 
	\item $3/2 \leq \overline{s}_2 < 3/2+m$ and
		\begin{equation*}
			\overline{s}_2 
			\leq \overline{\alpha}_2 
			\leq \overline{s}_2 \cdot \frac{d}{d-1} \cdot\frac{m}{3/2+m-\overline{s}_2}.
		\end{equation*}
	\item $3/2+m \leq \overline{s}_2 \leq \overline{\alpha}_2$.
\end{enumerate}
For the regularity of the pressure $\pi$ an analogous statement holds with $3/2$ replaced by $1/2$.
\end{theorem}

\begin{proof}
Let us only consider the assertions on $S_u(\Omega)$; the results for $S_\pi(\Omega)$ can be derived 
%exactly in 
 in exactly 
the same way.
Due to \autoref{prop:embeddings}\ref{it:embedd:Sob:eps} and \autoref{lem:Stokes} applied for $p=2$ we know that 
\begin{equation*}
	S_u(\Omega) \subseteq  H^{s+1/2}(\Omega)^d = F_{2,2}^{s+1/2}(\Omega)^d 
	\qquad \text{for all}\qquad
	s<1.
\end{equation*}
Therefore, by \autoref{rem:main}\ref{rem:main:q} we have
$3/2\leq \overline{s}_2 \leq \overline{\alpha}_2$.

Since $m>0$, it remains to show that if $\overline{s}_2 < 3/2+m$, then the stated upper bound on $\overline{\alpha}_2$ holds true. 
To this end, let us define
\begin{equation*}
	\overline{\delta} := \min\left\{1,\, \frac{1}{2} \left( \frac{3}{2} + m - \overline{s}_2 \right)\right\}.
\end{equation*}
Then $3/2\leq \overline{s}_2 < 3/2+ m$ particularly implies that $0<\overline{\delta} < m \leq (d-1)/2$.
For each arbitrarily fixed $\delta \in (0,\overline{\delta})$ we can now choose $p_z=p_z(\delta) \in (2,\infty)$ with
\begin{equation*}
	(d-1)\left( \frac{1}{2}-\frac{1}{p_z} \right) = m - \delta.
\end{equation*}
Then the definition of $m$ implies that
\begin{equation*}
	0<\frac{1}{2}-\frac{1}{p_z} < \min\left\{ \frac{\epsilon(\Omega)}{2}, \frac{\sigma}{d-1}\right\}
\end{equation*}
and hence $p_z$ satisfies \eqref{cond:p}. 
Thus, \autoref{lem:Stokes} ensures that $S_u(\Omega)\subseteq H^s_{p_z}(\Omega)=F^{s}_{p_z,2}(\Omega)$ for all $s<z:=1+1/p_z$. According to \autoref{rem:sufficient}, this allows to apply \autoref{thm:main:new}, where
\begin{equation*}
	\mu = \overline{s}_2 - d\, \left( \frac{1}{2}-\frac{1}{p_z}\right) = z + \overline{s}_2 - \frac{3}{2} - m + \delta < z - \frac{1}{2} \left( \frac{3}{2}+m - \overline{s}_2\right) < z.
\end{equation*}
Therefore, the bound~\eqref{eq:alpha:bound} applies which shows that
\begin{equation*}
	\overline{\alpha}_2 \leq \overline{s}_2 \cdot \frac{\overline{s}_2-\mu}{z-\mu} 
	= \overline{s}_2 \cdot \frac{d\,(1/2-1/p_z)}{3/2+  m-\delta - \overline{s}_2} 
	= \overline{s}_2 \cdot \frac{d}{d-1} \cdot \frac{m-\delta}{3/2+  m-\delta - \overline{s}_2}.
\end{equation*}
Since the latter inequality holds for arbitrary small $\delta>0$, this completes the proof.
\end{proof}

Let us conclude also this section with some final remarks:

\begin{remark}
Assume for simplicity that $\sigma = \sigma_1$ is chosen small enough such that $m=\sigma$. Then case~2.)~in \autoref{thm:stokes} can be interpreted as a shift $H^{-1/2+\sigma} \ni f \mapsto u \in H^{3/2+\sigma}$ of full order (two) within the Sobolev scale. However, as we have seen in \autoref{sect:Poisson}, already for the classical Poisson problem this shift might fail even on $\cont^1$ domains. Although we do not know about an explicit example, it is very likely that the same is true for the Stokes problem. Then case~1.)~applies and we have a non-trivial upper bound $\overline{\alpha}_2 \leq b$ on the Besov smoothness w.r.t.\ the scale \eqref{eq:scale} with $p=2$. Moreover note that this $b=b(\overline{s}_2)$ is monotonically increasing in $\overline{s}_2$, where 
\begin{equation*}
	\frac{3}{2} \, \frac{d}{d-1} = b(3/2) \leq b(\overline{s}_2) < b(3/2+m) = \infty, \qquad \overline{s}_2\in [3/2,3/2+m).
\end{equation*}

Recently Eckhardt et al.~\cite[Theorem~3.3]{EckCioDah2016} addressed the question of Besov regularity for dimensions $d\geq 3$ under the additional conditions that the boundary of $\Omega$ is connected and $g=0$. Rewritten in our notation they were able to show that for $\sigma_1=1/2$ and $\sigma_3=0$ we have for $d\geq 4$
\begin{equation*}
	\overline{\alpha}_2(S_u(\Omega))\geq \frac{3}{2}\, \frac{d}{d-1}
	\qquad \text{and} \qquad 
	\overline{\alpha}_2(S_\pi(\Omega))\geq \frac{1}{2}\, \frac{d}{d-1}.
\end{equation*}
\end{remark}

\begin{appendix}
\section{Appendix: Basics from function space theory}\label{sect:appendix}
In this supplementary section we collect the main definitions and assertions concerning function spaces on domains which are needed throughout the paper. Here `domain' always means `non-empty, connected, open set'. Special attention is paid to bounded Lipschitz domains $\Omega\subseteq\R^d$, $d\in\N$, as defined, e.g., in Triebel~\cite[Section~1.11.4]{T06}.

\subsection{Besov and Triebel-Lizorkin spaces}
In accordance with Triebel~\cite{T83} we use the Fourier analytic approach towards Besov and Triebel-Lizorkin spaces on $\R^d$ and define the corresponding spaces on domains by restriction.

Let $d\in\N$. By $\S(\R^d)$ we denote the Schwartz space of all complex-valued rapidly decreasing $\cont^\infty$ functions on $\R^d$ and $\S'(\R^d)$ denotes its dual space of tempered distributions.
Moreover, for domains $\Omega\subseteq\R^d$ we let $\D(\Omega):=\cont_0^\infty(\Omega)$ denote the collection of all complex-valued $\cont^\infty$ functions in $\R^d$ with compact support in $\Omega$ and denote by $\D'(\Omega)$ its dual space of distributions on $\Omega$. As usual, we say two functionals $f$ and $g$ equal each other in $\S'(\R^d)$ or~$\D'(\Omega)$ if
$$
f(\varphi)=g(\varphi) \qquad \text{for all }\, \varphi \,\text{ from }\, \S(\R^d) \,\text{ or }\, \D(\Omega), \,\text{ respectively}.
$$
For $g\in \S'(\R^d)$ we denote by $g_{|_{\Omega}}$ the restriction of $g$ to $\Omega$ which means that
$$
g_{|_{\Omega}} \in \D'(\Omega) \qquad \text{and} \qquad (g_{|_{\Omega}})(\varphi):=g(\varphi) \qquad \text{for all} \qquad \varphi\in \D(\Omega).
$$
Note that this is meaningful since $\D(\Omega)\subseteq\D(\R^d)\subseteq\S(\R^d)$.

In addition, let $\F$ and $\F^{-1}$ denote the (extension of the) Fourier transform, respectively its inverse, on $\S'(\R^d)$.
Fix an arbitrary $\phi_0 \in \S(\R^d)$ such that
$$
\phi_0(x)=1 \quad \text{if} \quad \abs{x}_2\leq 1 \qquad \text{and} \qquad \phi_0(x)=0 \quad \text{if} \quad \abs{x}_2\geq \frac{3}{2}.
$$
Then the collection $\Phi:=(\phi_k)_{k\in\N_0}$, with 
$$
\phi_k(x):=\phi_0(2^{-k}x)-\phi_0(2^{-k+1}x), \qquad x\in\R^d,\qquad k\in\N,
$$
defines a smooth dyadic resolution of unity and we have
$$
f = \sum_{k=0}^\infty \F^{-1}[\phi_k\, \F f] \qquad \text{(convergence in $\S'(\R^d)$)}
$$
for all $f\in\S'(\R^d)$. Due to the celebrated Paley-Wiener-Schwartz-Theorem, the building blocks $\F^{-1}[\phi_k\, \F f]$, $k\in\N_0$, are actually entire analytic functions; see, for instance, Triebel~\cite[Section~1.2.1]{T83}.
As usual, for $0<q<\infty$, $\ell_q(\N_0)$ is the space of  $q$-summable scalar-valued sequences over $\N_0$ (bounded sequences, if $q=\infty$). 

\begin{defi}
For $d\in\N$ choose $\Phi$ as above and let $\Omega\subsetneq\R^d$ denote an arbitrary domain. Moreover, let $s\in\R$ and $0<p,q\leq \infty$.
\begin{enumerate}[align=right,label=\textup{(\roman*)}] 
\item The set $B^s_{p,q}(\R^d):=\left\{f\in \S'(\R^d) \sep \norm{f \sep B^s_{p,q}(\R^d)} <\infty \right\}$, quasi-normed by
$$
\norm{f \sep B^s_{p,q}(\R^d)} := \norm{ \left( 2^{ks} \norm{\F^{-1}[\phi_k\, \F f](\cdot) \sep L_p(\R^d)} \right)_{k\in\N_0} \sep \ell_q(\N_0)},
$$
is called Besov space.
\item If $p<\infty$, then the set $F^s_{p,q}(\R^d):=\left\{f\in \S'(\R^d) \sep \norm{f \sep F^s_{p,q}(\R^d)} <\infty \right\}$, quasi-normed by
$$
\norm{f \sep F^s_{p,q}(\R^d)} := \norm{   \norm{ \left( 2^{ks} \abs{\F^{-1}[\phi_k\, \F f](\cdot)}  \right)_{k\in\N_0} \sep \ell_q(\N_0)} \sep L_p(\R^d)},
$$
is called Triebel-Lizorkin space.
\item If $A\in\{B,F\}$ with $p<\infty$ for $A=F$, then the set
$$
A^s_{p,q}(\Omega):=\left\{ f \in \D'(\Omega) \sep \text{there exists } g\in A^{s}_{p,q}(\R^d) \text{ with } g_{|_{\Omega}} = f \text{ in } \D'(\Omega) \right\},
$$
quasi-normed by
$$
\norm{f \sep A^s_{p,q}(\Omega)}:= \inf_{\substack{g\in A^{s}_{p,q}(\R^d)\\ g_{|_{\Omega}} = f \text{ in } \D'(\Omega)}} \norm{g \sep A^{s}_{p,q}(\R^d)},
$$
is called Besov resp.\ Triebel-Lizorkin space on $\Omega$.
\end{enumerate}
\end{defi}
Standard proofs show that the spaces introduced above are  quasi-Banach spaces (Banach iff $\min\{p,q\}\geq 1$ and Hilbert iff $p=q=2$) and that different $\Phi$ provide equivalent quasi-norms, see, e.g., Triebel~\cite[Section~2.3.2]{T83}. 
Furthermore, these scales of spaces cover a variety of classical function spaces---such as, e.g.,
Lebesgue, Sobolev(-Slobodeckij), Bessel potential, Lipschitz, H\"older(-Zygmund), or Hardy spaces---as special cases. Besides our Fourier analytic definition, there is a big variety of other descriptions of these spaces which are equivalent at least for large ranges of parameters.
To give an example, we note that at least for
$$
s > \sigma_{p}:= d \, \max\left\{\frac{1}{p}-1,0\right\}
$$
the spaces~$A^s_{p,q}(\R^d)$ (and also $A^{s}_{p,q}(\Omega)$ for bounded Lipschitz domains $\Omega\subseteq\R^d$) exclusively contain regular distributions, i.e., functions, which makes it possible to characterize them as subspaces of some Lebesgue space by means of iterated differences. For details we refer to Triebel~\cite[Section~1.11.9]{T06}.

\subsection{Sobolev spaces}
We follow the usual approach and define the 
%subsequent 
 following  
Sobolev-type spaces based on Besov and Triebel-Lizorkin spaces.

\begin{defi}\label{def:spaces}
	For $d\in \N$ let $\Omega\subseteq\R^d$ denote a bounded Lipschitz domain. 
	Then we set
	\begin{align*}
&W^m_p(\Omega):=F^m_{p,2}(\Omega), & m\in\N_0, 1<p<\infty, & \quad\text{(Sobolev)}\\
&W^s_p(\Omega):=F^s_{p,p}(\Omega)=B^s_{p,p}(\Omega), & 0<s\notin\N, 1\leq p < \infty, & \quad\text{(Sobolev-Slobodeckij)}\\
&W^s_p(\Omega):=\left[ W^{-s}_{p',0}(\Omega) \right]', & s<0, 1<p<\infty,  & \\
&H^s_p(\Omega):=F^s_{p,2}(\Omega), & s\in\R, 1<p<\infty, & \quad\text{(Bessel potential)}\\
&H^s(\Omega):=H^s_2(\Omega)=F^s_{2,2}(\Omega)=B^s_{2,2}(\Omega), & s\in\R, & \quad\text{(Sobolev-Hilbert)}
\end{align*}
where for $1<p<\infty$ the index $p'$ is given by $1/p+1/p'=1$ and $W_{p,0}^s(\Omega)$ denotes the closure of $\cont_0^\infty(\Omega)$ w.r.t.\ the norm $\norm{\cdot \sep W_p^s(\Omega)}$ if $s>0$.
\end{defi}

It is worth noting that these definitions are equivalent with the common definitions of Sobolev(-Slobodeckij) and Bessel potential %spaces. In particular, for 
 spaces: For 
$s=m\in\N_0$ we have
$$
W^m_p(\Omega)=\bigg\{f\in L_p(\Omega) \,\bigg|\, \norm{f \sep W^m_p(\Omega)}:=\bigg[ \sum_{\abs{\alpha}_1\leq m} \norm{D^\alpha f \sep L_p(\Omega)}^p \bigg]^{1/p}<\infty \bigg\},
$$
see Triebel~\cite[Theorem~1.122]{T06},
while $W^s_p(\Omega)=B^s_{p,p}(\Omega)$ for $0<s\notin\N$ coincides with the 
%classical 
definition of Sobolev-Slobodeckij spaces as real interpolation space of $L_p(\Omega)$ with~$W^m_p(\Omega)$ for some 
 $m\in \N$ with $m>s$ 
%$s<m \in\N$ 
and suitable parameters; see, e.g., DeVore~\cite[Section~4.6]{DeV1998}.

\subsection{Embeddings}
The scales of Besov and Triebel-Lizorkin spaces $A^s_{p,q}(\Omega)$ on bounded Lipschitz domains satisfy various embeddings. Let us mention a few of them:

\begin{prop}\label{prop:embeddings}
For $d\in\N$ let $\Omega\subseteq\R^d$ denote a bounded Lipschitz domain. Further assume $s,s_0,s_1\in\R$ and let $0<p,p_0,p_1,q,q_0,q_1\leq \infty$.
\begin{enumerate}[align=right,label=\textup{(\roman*)}] 
	\item Assume additionally that $p<\infty$. Then 
		$$
			B^{s}_{p,q_0}(\Omega) \hookrightarrow F^{s}_{p,q}(\Omega) \hookrightarrow B^{s}_{p,q_1}(\Omega).
		$$
		holds if, and only if, we have $q_0 \leq \min\{p,q\} \leq \max\{p,q\}\leq q_1$.
	\item\label{it:embedd:Sob} If additionally $p_0<p_1<\infty$ and $s_0-d/p_0=s_1-d/p_1$, then
		$$
			F^{s_0}_{p_0,q_0}(\Omega) \hookrightarrow F^{s_1}_{p_1,q_1}(\Omega).
		$$
\item If additionally $A\in\{B,F\}$ (and $p<\infty$ if $A=F$), as well as $q_0\leq q_1$, then
$$
A^s_{p,q_0}(\Omega) \hookrightarrow A^{s}_{p,q_1}(\Omega).
$$
\item\label{it:embedd:Sob:eps} If additionally $X,Y\in\{B,F\}$ and 
$$
s_0-s_1 > d \, \max\left\{\frac{1}{p_0} -\frac{1}{p_1}, 0 \right\},
$$
then
$$
X^{s_0}_{p_0,q_0}(\Omega) \hookrightarrow Y^{s_1}_{p_1,q_1}(\Omega)
$$
(with finite integrability parameter for $F$-spaces).
\item\label{it:embedd:sob:sharp} Assume additionally that $p_0<p<p_1$ and
\begin{equation*}
s_0-\frac{d}{p_0} = s-\frac{d}{p} =s_1-\frac{d}{p_1}.
\end{equation*}
Then
\begin{equation*}
B^{s_0}_{p_0,q_0}(\Omega)
\hookrightarrow
F^s_{p,q}(\Omega)
\hookrightarrow
B^{s_1}_{p_1,q_1}(\Omega)
\end{equation*}
holds if, and only if, we have $q_0\leq p\leq q_1$.
\end{enumerate}
\end{prop}
\begin{proof}
For (i), (ii), and (v) see, e.g., Triebel~\cite[page 60]{T06} and the references therein. For (iii) and (iv) additionally consult Triebel~\cite[Proposition~2 in Section~2.3.2]{T83}, as well as \cite[Theorem~4.33 and Remark~4.34]{T08}.
%Hansen, Sickel~\cite[Sect.~2.5]{HanSic2011} (beware of the typo!).
\end{proof}

Note that \autoref{prop:embeddings}(iv) particularly implies that for $A\in\{B,F\}$ we have
$$
A^{s_0}_{p_0,q}(\Omega) \hookrightarrow W^{s_1}_{p_1}(\Omega) 
\qquad \text{if} \qquad 
s_0>s_1\geq 0, \text{ as well as } 1<p_1 \leq p_0 \leq \infty, \text{ and } 0<q\leq\infty
$$
with $p_0<\infty$ if $A=F$, since $W^{s_1}_p(\Omega)$ can be identified with 
%$F^{s_1}_{p,q}(\Omega)$ for $q\in\{2,p\}$.
 $F^{s_1}_{p,2}(\Omega)$ (if $s_1\in\N$) or $F^{s_1}_{p,p}(\Omega)$ (if $0<s_1 \notin\N$).

\subsection{Complex interpolation}
For some open set $\Omega$ let $X(\Omega)$ and $Y(\Omega)$ denote quasi-normed spaces of complex-valued functions or distributions on $\Omega$. Then, under certain conditions,  the (extended) \emph{complex interpolation method} is applicable and yields further quasi-normed spaces of functions on $\Omega$. 
% the application of the so-called (extended) \emph{complex interpolation method} with parameter $\theta\in(0,1)$ defines another quasi-normed space of functions on $\Omega$. 
Besides other useful properties  these spaces,  
%this space,
usually denoted by $[X(\Omega),Y(\Omega)]_\theta$, 
 $\theta\in(0,1)$, satisfy
%satisfies
$$
X(\Omega)\cap Y(\Omega) \hookrightarrow [X(\Omega),Y(\Omega)]_\theta 
\hookrightarrow X(\Omega)+ Y(\Omega).
$$
 Thus, in particular, any set $S(\Omega)\subset X(\Omega)\cap Y(\Omega)$ is also contained in $[X(\Omega),Y(\Omega)]_\theta$ for all $\theta\in (0,1)$. 
For details we refer to Bergh, L\"ofstr\"om~\cite{BL76} and Kalton, Mayboroda, Mitrea~\cite{KMM07}.

It turns out that the scales of Besov and Triebel-Lizorkin spaces $A^s_{p,q}(\Omega)$ on bounded Lipschitz domains behave well w.r.t.\ this method:
\begin{prop}[{Kalton et al.~\cite[Theorem~9.4]{KMM07}}]\label{prop:interpol}
For $d\in\N$ let $\Omega\subseteq\R^d$ denote a bounded Lipschitz domain and assume $\theta\in(0,1)$. Moreover, let $A\in\{B,F\}$, as well as $s,s_0,s_1\in\R$, and $0<p,p_0,p_1,q,q_0,q_1\leq\infty$ (with $p_0,p_1<\infty$ for $A=F$), and $\min\{q_0,q_1\}<\infty$. Then
$$
s=(1-\theta)\,s_0 + \theta\, s_1, \qquad 
\frac{1}{p}=\frac{1-\theta}{p_0}+\frac{\theta}{p_1}, \qquad \text{and}\qquad \frac{1}{q}=\frac{1-\theta}{q_0}+\frac{\theta}{q_1}
$$
implies
$$
\left[ A^{s_0}_{p_0,q_0}(\Omega), A^{s_1}_{p_1,q_1}(\Omega) \right]_{\theta} = A^{s}_{p,q}(\Omega)
$$
in the sense of equivalent quasi-norms.
\end{prop}

\subsection*{Acknowledgements}
The authors are grateful to the anonymous reviewers  for their valuable comments and their constructive suggestions which  helped to improve the  manuscript.

\end{appendix}

\addcontentsline{toc}{chapter}{References}
\bibliographystyle{is-abbrv}

\small
%\bibliography{Bibliography}

\begin{thebibliography}{10}
\setlength{\parskip}{-3pt}

\bibitem{BalDie+2019}
A.~Kh.~Balci, L.~Diening, and M.~Weimar.
\newblock Higher order {C}alder{\'o}n-{Z}ygmund estimates for the $p$-{L}aplace
  equation.
\newblock {\em J.~Differential Equations}, 268:\penalty0 590--635, 2020.
%\newblock doi:10.1016/j.jde.2019.08.009.


\bibitem{BL76}
J.~Bergh and J.~L{\"o}fstr{\"o}m.
\newblock {\em {Interpolation {S}paces. {A}n {I}ntroduction}}.
\newblock Grundlehren der mathematischen Wissenschaften. Springer-Verlag,
  Berlin, 1976.

\bibitem{CDK+}
P.~Cioica, S.~Dahlke, S.~Kinzel, F.~Lindner, T.~Raasch, K.~Ritter, and
  R.~Schilling.
\newblock {S}patial {B}esov regularity for stochatic partial differential
  equations on {L}ipschitz domains.
\newblock {\em Studia Math.}, 207\penalty0 (3):\penalty0 197--234, 2011.

\bibitem{Cos2019}
M.~Costabel.
\newblock On the limit {S}obolev regularity for {D}irichlet and {N}eumann
  problems on {L}ipschitz domains.
%\newblock Preprint, 2017.
%\newblock arXiv:1711.07179.
\newblock {\em Math.\ Nachr.}, 292:2165--2173, 2019.


\bibitem{Dah1999}
S.~Dahlke.
\newblock {B}esov regularity for elliptic boundary value problems on polygonal
  domains.
\newblock {\em Appl.\ Math.\ Lett.}, 12:\penalty0 31--38, 1999.

\bibitem{DahDahDeV1997}
S.~Dahlke, W.~Dahmen, and R.~A. DeVore.
\newblock Nonlinear approximation and adaptive techniques for solving elliptic
  operator equations.
\newblock In W.~Dahmen, A.~Kurdila, and P.~Oswald, editors, {\em {M}ultsicale
  {W}avelet {M}ethods for {P}artial {D}ifferential {E}quations}, pages
  237--283, San Diego, 1997. Academic Press.

\bibitem{DahDeV1997}
S.~Dahlke and R.~A. DeVore.
\newblock Besov regularity for elliptic boundary value problems.
\newblock {\em Comm.\ Partial Differential Equations}, 22\penalty0
  (1-2):\penalty0 1--16, 1997.

\bibitem{DahDieHar+2014}
S.~Dahlke, L.~Diening, C.~Hartmann, B.~Scharf, and M.~Weimar.
\newblock Besov regularity of solutions to the $p$-{P}oisson equation.
\newblock {\em Nonlinear Anal.}, 130:\penalty0 298--329, 2016.

\bibitem{DahHar+2018}
S.~Dahlke, H.~Harbrecht, M.~Utzinger, and M.~Weimar.
\newblock Adaptive wavelet {BEM} for boundary integral equations: {T}heory and
  numerical experiments.
\newblock {\em Numer.\ Funct.\ Anal.\ Optim.}, 39\penalty0 (2):\penalty0
  208--232, 2018.

\bibitem{DaSi08}
S.~Dahlke and W.~Sickel.
\newblock Besov regularity for the poisson equation in smooth and polyhedral
  cones.
\newblock In V.~Maz'ya, editor, {\em {S}obolev {S}paces in {M}athematics {II},
  {A}pplications to {P}artial {D}ifferential {E}quations}, pages 123--146.
  Springer, 2008.

\bibitem{DaSi13}
S.~Dahlke and W.~Sickel.
\newblock On {B}esov regularity of solutions to nonlinear elliptic partial
  differential equations.
\newblock {\em Rev.\ Mat.\ Complut.}, 26\penalty0 (1):\penalty0 115--145, 2013.

\bibitem{DahWei2015}
S.~Dahlke and M.~Weimar.
\newblock Besov regularity for operator equations on patchwise smooth
  manifolds.
\newblock {\em J.~Found.\ Comput.\ Math.}, 15\penalty0 (6):\penalty0
  1533--1569, 2015.

\bibitem{DeV1998}
R.~A. DeVore.
\newblock Nonlinear approximation.
\newblock {\em Acta Numer.}, 7:\penalty0 51--150, 1998.

\bibitem{Ebm2002}
C.~Ebmeyer.
\newblock Mixed boundary value problems for nonlinear elliptic systems with
  p-structure in polyhedral domains.
\newblock {\em Math. Nachr.}, 236:\penalty0 91--108, 2002.

\bibitem{Eck2015}
F.~Eckhardt.
\newblock Besov regularity for the {S}tokes and the {N}avier-{S}tokes system in
  polyhedral domains.
\newblock {\em ZAMM - J.~Appl.\ Math.\ Mech.\ / Zeitschrift f{\"{u}}r Angew.\
  Math.\ und Mech.}, 95\penalty0 (11):\penalty0 1161--1173, 2015.

\bibitem{EckCioDah2016}
F.~Eckhardt, P.~A. Cioica-Licht, and S.~Dahlke.
\newblock Besov regularity for the stationary {N}avier-{S}tokes equation on
  bounded {L}ipschitz domains.
\newblock {\em Appl.\ Anal.}, 97\penalty0 (3):\penalty0 466--485, 2018.

\bibitem{FabMenMit1998}
E.~Fabes, O.~Mendez, and M.~Mitrea.
\newblock Boundary layers on {S}obolev-{B}esov spaces and {P}oisson's equation
  for the {L}aplacian on {L}ipschitz domains.
\newblock {\em J.~Funct.\ Anal.}, 159:\penalty0 323--368, 1998.

\bibitem{GasMor2014}
F.~D. Gaspoz and P.~Morin.
\newblock Approximation classes for adaptive higher order finite element
  approximation.
\newblock {\em Math.\ Comp.}, 83:\penalty0 2127–--2160, 2014.

\bibitem{Gri1985}
P.~Grisvard.
\newblock {\em Elliptic {P}roblems in {N}onsmooth {D}omains}.
\newblock Mongr.\ Stud.\ Math.~24. Pitman, Boston/London/Melbourne, 1985.

\bibitem{Gri1992}
P.~Grisvard.
\newblock {\em Singularities in {B}oundary {V}alue {P}roblems}.
\newblock Recherches en math{\'e}matiques appliqu{\'e}es~22. Springer,
  Paris/Berlin, 1992.

\bibitem{Han2015}
M.~Hansen.
\newblock Nonlinear approximation rates and {B}esov regularity for elliptic
  {PDE}s on polyhedral domains.
\newblock {\em J.~Found.\ Comput.\ Math.}, 15\penalty0 (2):\penalty0 561--589,
  2015.

\bibitem{HarWei2018}
C.~Hartmann and M.~Weimar.
\newblock Besov regularity of solutions to the $p$-{P}oisson equation in the
  vicinity of a vertex of a polygonal domain.
\newblock {\em Results Math.}, 73\penalty0 (41):\penalty0 1--28, 2018.

\bibitem{JK95}
D.~S. Jerison and C.~E. Kenig.
\newblock The inhomogeneous {D}irichlet problem in {L}ipschitz domains.
\newblock {\em J.~Funct.\ Anal.}, 130\penalty0 (1):\penalty0 161--219, 1995.

\bibitem{KMM07}
N.~Kalton, S.~Mayboroda, and M.~Mitrea.
\newblock Interpolation of {H}ardy-{S}obolev-{B}esov-{T}riebel-{L}izorkin
  spaces and applications to problems in partial differential equations.
\newblock In L.~{De Carli} and M.~Milman, editors, {\em Interpolation {T}heory
  and {A}pplications ({C}ontemporary {M}athematics 445)}, pages 121--177,
  Providence, RI, 2007. Amer. Math. Soc.

\bibitem{Lin2014}
F.~Lindner.
\newblock Singular behavior of the solution to the stochastic heat equation on
  a polygonal domain.
\newblock {\em Stoch.\ Partial Differ.\ Equ.\ Anal.\ Comput.}, 2\penalty0
  (2):\penalty0 146--195, 2014.

\bibitem{May2005}
S.~Mayboroda.
\newblock {\em The {P}oisson {P}roblem on {L}ipschitz {D}omains}.
\newblock PhD thesis, University of Missouri-Columbia, 2005.

\bibitem{MazRos2010}
V.~G. Maz'ya and J.~Ro{\ss}mann.
\newblock {\em Elliptic {E}quations in {P}olyhedral {D}omains}.
\newblock Math.\ Surveys Monogr.~162. Amer.\ Math.\ Soc., Providence, RI, 2010.

\bibitem{MitWri2012}
M.~Mitrea and M.~Wright.
\newblock Boundary value problems for the {S}tokes system in arbitrary
  {L}ipschitz domains.
\newblock {\em Ast{\'e}risque}, 344:\penalty0 1--241, 2012.

\bibitem{RunSic1996}
T.~Runst and W.~Sickel.
\newblock {\em {S}obolev {S}paces of {F}ractional {O}rder, {N}emytskij
  {O}perators and {N}onlinear {P}artial {D}ifferential {E}quations}.
\newblock de Gruyter, Berlin/New York, 1996.

\bibitem{Sav1998}
G.~Savar{\'e}.
\newblock Regularity results for elliptic equations in {L}ipschitz domains.
\newblock {\em J.~Funct.\ Anal.}, 152:\penalty0 176--201, 1998.

\bibitem{T83}
H.~Triebel.
\newblock {\em Theory of {F}unction {S}paces}.
\newblock Birkh\"auser, Basel/Boston/Stuttgart, 1983.

\bibitem{T06}
H.~Triebel.
\newblock {\em Theory of {F}unction {S}paces III}.
\newblock Birkh\"auser, Basel, 2006.

\bibitem{T08}
H.~Triebel.
\newblock {\em Function {S}paces and {W}avelets on {D}omains}, volume~7 of {\em
  EMS Tracts in Mathematics}.
\newblock European Mathematical Society (EMS), Z\"urich, 2008.

\bibitem{Wei2016}
M.~Weimar.
\newblock Almost diagonal matrices and {B}esov-type spaces based on wavelet
  expansions.
\newblock {\em J.~Fourier Anal.\ Appl.}, 22\penalty0 (2):\penalty0 251--284,
  2016.

\end{thebibliography}

\end{document}